\documentclass{amsart}

\usepackage{amsfonts,amsmath,amssymb}
\newcommand{\bydef}{:=}

\newcommand{\bg}{{\overline{g}}}
\newcommand{\bt}{{\overline{t}}}

\newcommand{\bG}{{\overline{G}}}
\newcommand{\tV}{{\widetilde{V}}}
\newcommand{\tW}{{\widetilde{W}}}
\newcommand{\tB}{{\widetilde{B}}}
\newcommand{\bT}{{\overline{T}}}
\newcommand{\ov}{\overline}

\newcommand{\wh}[1]{\widehat{#1}}
\newcommand{\wt}[1]{\widetilde{#1}}

\newcommand{\vphi}{\varphi}
\newcommand{\veps}{\varepsilon}
\newcommand{\ve}{\varepsilon}
\newcommand{\vp}{\varphi}
\newcommand{\ld}{\ldots}

\newcommand{\sgn}{\mathrm{sgn}}
\DeclareMathOperator*{\ot}{\otimes}%allows placement of subscript below in displaymath

\newcommand{\ontop}[2]{\genfrac{}{}{0pt}{}{#1}{#2}}%like \frac, but without horizontal line

\newcommand{\sks}{\mathcal{K}}%skew Hermitian (skew symmetric)
\newcommand{\id}{\mathrm{id}}%identity map

\newcommand{\cA}{\mathcal{A}}%algebras

\newcommand{\cC}{\mathcal{C}}
\newcommand{\cD}{\mathcal{D}}

\newcommand{\cH}{\mathcal{H}}

\newcommand{\cL}{\mathcal{L}}
\newcommand{\cM}{\mathcal{M}}
\newcommand{\cQ}{\mathcal{Q}}
\newcommand{\cR}{\mathcal{R}}
\newcommand{\cS}{\mathcal{S}}
\newcommand{\cU}{\mathcal{U}}

\newcommand{\frA}{\mathfrak{A}}
\newcommand{\frB}{\mathfrak{B}}
\newcommand{\frC}{\mathfrak{C}}
\newcommand{\frF}{\mathfrak{F}}

\newcommand{\NN}{\mathbb{N}}

\newcommand{\RR}{\mathbb{R}}

\newcommand{\CC}{\mathbb{C}}

\newcommand{\FF}{\mathbb{F}}

\newcommand{\chr}[1]{\mathrm{char}\,#1}

\newcommand{\Homgr}{\mathrm{Hom}^\mathrm{gr}}%graded space of homomorphisms (subspace of Hom)
\DeclareMathOperator{\End}{\mathrm{End}}
\newcommand{\Endgr}{\mathrm{End}^\mathrm{gr}}%graded space of endomorphisms (subspace of End)

\DeclareMathOperator{\Aut}{\mathrm{Aut}}%automorphism group
\newcommand{\brac}[1]{{#1}^{(-)}}%associative algebra as a Lie algebra

\newcommand{\Sl}{\mathfrak{sl}}
\newcommand{\fsl}{\mathfrak{fsl}}%finitary

\newcommand{\So}{\mathfrak{so}}
\newcommand{\fso}{\mathfrak{fso}}%finitary
\newcommand{\Sp}{\mathfrak{sp}}
\newcommand{\fsp}{\mathfrak{fsp}}%finitary

\newtheorem{theorem}{Theorem}[section]

\newtheorem{lemma}[theorem]{Lemma}
\newtheorem{corollary}[theorem]{Corollary}
\newtheorem{proposition}[theorem]{Proposition}

\theoremstyle{definition}

\newtheorem{remark}[theorem]{Remark}
\newtheorem{df}[theorem]{Definition}

\newcommand{\gr}{\mathrm{gr}}
\newcommand{\CF}[2]{\mathfrak{F}_{#2}^\gr(#1)}
\newcommand{\CL}[2]{\mathfrak{L}_{#2}^\gr(#1)}
\newcommand{\fvw}{\CF{V}{W}}
\newcommand{\lvw}{\CL{V}{W}}
\newcommand{\fpu}{\mathfrak{F}_\Pi(U)}
\newcommand{\Q}{\mathfrak{Q}}
\newcommand{\R}{\mathfrak{R}}

\begin{document}

\title{Group gradings on finitary simple Lie algebras}

\author[Bahturin]{Yuri Bahturin}
\address{Department of Mathematics and Statistics, Memorial
University of Newfoundland, St. John's, NL, A1C5S7, Canada}
\email{bahturin@mun.ca}

\author[Bre\v{s}ar]{Matej Bre\v{s}ar}
\address{Faculty of Mathematics and Physics, University of Ljubljana, Slovenia, and 
Faculty of Natural Sciences and Mathematics, University of Maribor,  Slovenia}
\email{matej.bresar@fmf.uni-lj.si}

\author[Kochetov]{Mikhail Kochetov}
\address{Department of Mathematics and Statistics, Memorial
University of Newfoundland, St. John's, NL, A1C5S7, Canada}
\email{mikhail@mun.ca}

\thanks{{\em Keywords:} graded algebra, simple Lie algebra, grading, primitive algebra, functional identity}
\thanks{{\em 2000 Mathematics Subject Classification:} Primary 17B70; Secondary 16R60, 16W50, 17B60.}
\thanks{The first author acknowledges support by by NSERC grant \# 227060-09. The second author acknowledges support by ARRS grant \#  P1-0288. The third author acknowledges support by NSERC grant \# 341792-07. We all are thankful to the referee for sharing with us some interesting ideas.}

\begin{abstract}
We classify, up to isomorphism, all gradings by an arbitrary abelian group on simple finitary Lie algebras of linear transformations (special linear, orthogonal and symplectic) on infinite-dimensional vector spaces over an algebraically closed field of characteristic different from $2$.
\end{abstract}

\maketitle

%-------------------------------------------------------------------------------

\section{Introduction}\label{intro}

The paper \cite{BZ10}, which appeared several years after it was written, describes all gradings by a finite abelian group $G$  on the associative algebra $\cR=M_\infty(\FF)$ of (countably) infinite matrices with finitely many nonzero entries from an algebraically closed field $\FF$ of characteristic zero. The ultimate goal was to apply these results in a later paper in order to describe all gradings by finite abelian groups on the Lie algebra $\cL$ isomorphic to one of $\Sl(\infty)$, $\So(\infty)$ and $\Sp(\infty)$. 

The tool that, in principle, allowed a reduction of $G$-gradings on $\cL$ to $G$-gradings on $\cR$ was developed in \cite{BB} and involved the so called functional identities (see \cite{FIbook}). However, for an arbitrary abelian group $G$, the results in \cite{BB} allowed such reduction only for unital associative algebras, and $\cR$ is not such. Hence, a considerable portion of this paper (Section \ref{sFILH}) is devoted to extending the results of \cite{BB} to the non-unital case, thus opening the road to the gradings on infinite-dimensional simple Lie algebras.

At the same time, using Hopf algebra approaches, \cite{BKM} and \cite{BK} classified, up to isomorphism, all gradings by an abelian group $G$ on finite-dimensional simple Lie algebras of types $A$, $B$, $C$ and $D$ (except $D_4$) over algebraically closed fields of characteristic different from $2$. Also, a closer inspection of the older sources like \cite{JSR}, \cite{NVO82} and \cite{NVO} made it possible to find a smooth and conceptual approach to gradings on matrix algebras and classical Lie algebras of the said types (to appear in the monograph \cite{EK}).

Based on all these results and ideas, in this paper we not only achieve the goal stated at the beginning, but rather obtain the complete classification, up to isomorphism, of gradings by an arbitrary abelian group $G$ on simple Lie algebras of linear operators of finite rank (special linear, orthogonal and symplectic) on an infinite-dimensional vector space over an algebraically closed field of characteristic different from $2$.

An infinite-dimensional simple Lie algebra over a field $\FF$ is called {\em finitary} if it can be represented by linear operators of finite rank. A complete classification of finitary simple Lie algebras was given in \cite{Bar} over any $\FF$ with $\chr{\FF}=0$ and in \cite{BaSt} over algebraically closed $\FF$ with $\chr{\FF}\ne 2,3$. Under the latter assumption, finitary simple Lie algebras over $\FF$ can be described in the following way. Let $U$ be an infinite-dimensional vector space over $\FF$. Let $\Pi\subset U^\ast$ be a {\em total} subspace, i.e., for any $v\ne 0$ in $U$ there is $f\in \Pi$ such that $f(v)\ne 0$. Let $\mathfrak{F}_\Pi(U)$ be the space spanned by the linear operators of the form $v\ot f$, $v\in U$ and $f\in \Pi$, defined by $(v\ot f)(u)=f(u)v$ for all $u\in U$. It is known from \cite{JSR} that $\cR\bydef\mathfrak{F}_\Pi(U)$ is a (non-unital) simple associative algebra. The commutator $[\cR,\cR]$ is a simple Lie algebra, which is denoted by $\fsl(U,\Pi)$. The algebra $\cR$ admits an ($\FF$-linear) involution if and only if there is a nondegenerate bilinear form $\Phi\colon U\times U\to\FF$ that identifies $U$ with $\Pi$ and that is either symmetric or skew-symmetric. The set of skew-symmetric elements with respect to this involution, i.e., the set of $r\in\cR$ satisfying $\Phi(ru,v)+\Phi(u,rv)=0$ for all $u,v\in U$, is a simple Lie algebra, which is denoted by $\fso(U,\Phi)$ if $\Phi$ is symmetric and $\fsp(U,\Phi)$ if $\Phi$ is skew-symmetric. In \cite{BaSt} it is shown that if $\FF$ is algebraically closed and $\chr{\FF}\ne 2,3$, then any finitary simple Lie algebra over $\FF$ is isomorphic to one of $\fsl(U,\Pi)$, $\fso(U,\Phi)$ or $\fsp(U,\Phi)$. The most important special case is that of countable (infinite) dimension. Then $U$ has countable dimension and the isomorphism class of the Lie algebra does not depend on $\Pi$ or $\Phi$. Hence, there are exactly three finitary simple Lie algebras of countable dimension: $\Sl(\infty)$, $\So(\infty)$ and $\Sp(\infty)$.

As we already stated, our goal is to classify all $G$-gradings on finitary simple Lie algebras. Following a number of previous papers, including \cite{BShZ, BZ06, BZ07, BKM, E, BK}, we try to obtain all $G$-gradings on a simple Lie algebra $\cL$ from $G$-gradings on a related associative algebra $\cR$. Here $\cL$ will be one of the algebras $\fsl(U,\Pi)$, $\fso(U,\Phi)$ or $\fsp(U,\Phi)$ described above, and the related associative algebra will be $\cR=\mathfrak{F}_\Pi(U)$ (with $\Pi=U$ in the last two cases). It is convenient to start with a more general setting --- namely, where $\cR$ is a primitive algebra with minimal one-sided ideals (see \cite[Chapter IV]{JSR}). In the case of finite-dimensional algebras over an arbitrary field $\FF$, these are just matrix algebras with coefficients in a finite-dimensional division algebra $\Delta$ (in particular, $\Delta=\FF$ if $\FF$ is algebraically closed). In general, there are more possibilities for $\cR$. As shown in \cite[IV.9]{JSR}, $\cR$ is a primitive algebra with minimal one-sided ideals if and only if it is isomorphic to a subalgebra of $\mathfrak{L}_\Pi(U)$ containing the ideal $\mathfrak{F}_\Pi(U)$ where $U$ is a right vector space over a division algebra $\Delta$, $\Pi$ is a total subspace of the left vector space $U^\ast$ over $\Delta$, $\mathfrak{L}_\Pi(U)$ is the algebra of all continuous $\Delta$-linear operators on $U$, and  $\mathfrak{F}_\Pi(U)$ is the set of all operators in $\mathfrak{L}_\Pi(U)$ whose image has finite dimension over $\Delta$. The term {\em continuous} refers here to the topology on $U$ with a neighbourhood basis at $0$ consisting of the sets of the form $\ker f_1\cap\ldots\cap\ker f_k$ where $f_1,\ldots,f_k\in \Pi$ and $k\in \NN$. A linear operator $A\colon U\to U$ is continuous with respect to this topology if and only if the adjoint operator $A^\ast\colon U^\ast\to U^\ast$ leaves the subspace $\Pi$ invariant. (In particular, if $\Delta$ is $\RR$ or $\CC$, $U$ is a Banach space and $\Pi$ consists of all bounded linear functionals on $U$, then a linear operator on $U$ is continuous in our sense if and only if it is bounded.)

In Section \ref{sSTGPA}, we will show that if an algebra $\cR$ as above is given a grading by a group $G$, then it becomes a graded primitive algebra with minimal one-sided graded ideals and hence there exists a graded division algebra $\cD$, a graded right vector space $V$ over $\cD$ and a total graded subspace $W$ of the graded dual $V^{\gr\ast}$ such that $\fvw\subset\cR\subset\lvw$ (Theorem \ref{theorem_graded_primitive}). Here, a {\em graded division algebra} is a $G$-graded unital associative algebra whose nonzero homogeneous elements are invertible, a {\em graded vector space} is a graded module over a graded division algebra (since such a module is necessarily free), $\lvw$ is the $\FF$-span of all homogeneous $\cD$-linear operators $A\colon V\to V$ such that $A^\ast(W)\subset W$, and $\fvw$ is the subset of operators whose image has finite dimension over $\cD$. Assuming $\FF$ algebraically closed, this result will lead us to a classification of $G$-gradings on $\fpu$ --- see Theorem \ref{lfd} and, for abelian $G$, Corollary \ref{lfd_abelian_G}. In the case of $M_\infty(\FF)$, the classification can be expressed in combinatorial terms --- see Corollary \ref{gradings_on_M_infinity}, which generalizes Theorems 5 and 6 from \cite{BZ10}.

In the study of gradings on $\cL=\fso(U,\Phi)$ or $\fsp(U,\Phi)$, we will need to classify involutions of the algebra $\fpu$ that preserve a grading. Also, in the study of what we call ``Type II'' gradings on $\cL=\fsl(U,\Pi)$, we will need to classify more general antiautomorphisms of $\fpu$. In these cases, $\fpu$ can be written as $\fvw$ where $W$ is identified with $V$ (up to a shift of grading) by means of a homogeneous sesquilinear form $B\colon V\times V\to\cD$. The classification of antiautomorphisms that we need is obtained by classifying such forms (Theorem \ref{lfd_2}).  

In Section \ref{sFILH}, we will use functional identities to describe any surjective Lie homomorphism $\rho$ from a Lie ideal of an associative algebra to $\cL\ot\cH$ (Theorem \ref{T1}) where $\cL$ is a noncentral ideal of $\brac{\cA}$, $\cA$ is an algebra of linear operators containing $\fpu$, and $\cH$ is a unital commutative associative algebra. Here we have used this  standard notation: given an associative algebra $\cU$, we denote by $\brac{\cU}$ the Lie algebra structure on $\cU$ defined by the commutator $[a,b]=ab-ba$. 

In the same section, we will also describe any surjective Lie homomorphism $\rho$ from an ideal of the Lie algebra of skew-symmetric elements in an associative algebra with involution to $\cL\ot\cH$ (Theorem \ref{T2}) where $\cL$ is a noncentral ideal of $\sks(\cA,\vphi)$, $\cH$ and $\cA$ are as above, $\vphi$ is an involution on $\cA$, and $\sks(\cA,\vphi)$ is the subalgebra of $\brac{\cA}$ consisting of all skew-symmetric elements with respect to $\vphi$. 

In Section \ref{sLie}, we will take $\cH$ to be the group algebra of the grading group $G$ and $\rho$ to be the extension $\cL\ot\cH\to\cL\ot\cH$ of a comodule algebra structure $\cL\to\cL\ot\cH$ by $\cH$-linearity. This will allow us to establish a link between $G$-gradings on the Lie algebra $\cL$ and those on the associative algebra $\cR$ where $\cL$ is one of $\fsl(U,\Pi)$, $\fso(U,\Phi)$ and $\fsp(U,\Phi)$ and $\cR=\fpu$. The classification of $G$-gradings on $\fsl(U,\Pi)$ is given by Theorem \ref{tmsla} and those on $\fso(U,\Phi)$ and $\fsp(U,\Phi)$ by Theorem \ref{tsosp}. In the case where $\cL$ has countable dimension, the classification can be stated in purely combinatorial terms --- see Corollary \ref{gradings_on_sl_infinity} for $\Sl(\infty)$ and Corollary \ref{gradings_on_sop_infinity} for $\So(\infty)$ and $\Sp(\infty)$.

%-------------------------------------------------------------------------------

\section{Generalities on gradings}\label{generalities}

Let $G$ be a group, written multiplicatively, with the identity element denoted by $e$. An abelian group $M$ (respectively, vector space over a field $\FF$) is said to be {\em $G$-graded} if $M$ has an abelian group (respectively, $\FF$-vector space) decomposition 
\[
M=\bigoplus_{g\in G} M_g.
\] 
Any such decomposition is called a {\em $G$-grading on $M$} or a {\em grading on $M$ by $G$}. The grading is {\em trivial} if $M=M_e$. $M_g$ is called the {\em homogeneous component} of degree $g$. For any $m\in M_g$, we will say that $m$ is {\em homogeneous of degree $g$} and write $\deg v=g$. The {\em support} of the $G$-grading is the set
\[
\{g\in G\;|\;M_g\neq 0\}.
\]

For a fixed $g\in G$, we define the left and right {\em shifts} ${}^{[g]}M$ and $M^{[g]}$ by setting ${}^{[g]}M_h\bydef M_{g^{-1}h}$ and $M^{[g]}_h\bydef M_{hg^{-1}}$ for all $h\in G$. In other words, if $m\in M$ is homogeneous of degree $h$, then $m$ will have degree $gh$ in ${}^{[g]}M$ and degree $hg$ in $M^{[g]}$.

A ring $\cU$ (respectively, algebra over a field $\FF$) is said to be $G$-graded if it is $G$-graded as an abelian group (respectively, vector space) such that, for all $g,h\in G$, we have $\cU_g \cU_h\subset \cU_{gh}$. We say that $\cU$ is {\em graded simple} if $\cU^2\ne 0$ and $\cU$ has no nonzero proper graded ideals.

If a $G$-graded abelian group (respectively, vector space) $M$ is a left module over a $G$-graded associative ring (respectively, algebra) $\cR$, then we say that $M$ is a graded module if, for all $g,h\in G$, we have $\cR_g M_h\subset M_{gh}$. Note that if $M$ is a graded left $\cR$-module and $g\in G$, then the right shift $M^{[g]}$ is again a graded left $\cR$-module, with the same action of $\cR$.

We say that a graded left $\cR$-module $M$ is {\em graded simple} if $\cR M\ne 0$ and $M$ has no nonzero proper graded submodules. The {\em graded socle} of $M$ is the sum of all graded simple submodules of $M$. $\cR$ is (left) {\em graded primitive} if it has a graded left module that is faithful and graded simple. $\cR$ is {\em graded prime} if, for any nonzero graded ideals $I$ and $J$ of $\cR$, we have $IJ\ne 0$. (It does not matter here if the ideals are left, right or two-sided.) 
%$\cR$ is {\em graded semiprime} if, for any nonzero graded ideal $I$, we have $I^2\ne 0$. 
It is easy to see that if $\cR$ is graded primitive, then it is graded prime.

A homomorphism (respectively, linear transformation) $f\colon M\to N$ of $G$-graded abelian groups (respectively, vector spaces) is said to be {\em homogeneous of degree $h$} if $f(M_g)\subset N_{hg}$, for all $g\in G$. Thus, the set $\Homgr(M,N)$ of finite sums of homogeneous maps from $M$ to $N$ is a $G$-graded abelian group (respectively, vector space). A homomorphism of graded abelian groups (respectively, vector spaces) is a homogeneous map of degree $e$.

We say that two $G$-gradings $\cU=\bigoplus_{g\in G} \cU_g$ and $\cU=\bigoplus_{g\in G} \cU'_g$ on a ring (respectively, algebra) $\cU$ are {\em isomorphic} if there exists an automorphism $\psi:\cU\to\cU$ such that
\[
\psi(\cU_g)=\cU'_{g}\;\mbox{ for all }\;g\in G,
\]
i.e., $\bigoplus_{g\in G} \cU_g$ and $\bigoplus_{g\in G} \cU'_g$ are isomorphic as $G$-graded rings (respectively, algebras).

%-------------------------------------------------------------------------------

\section{Graded primitive algebras with minimal graded left ideals}\label{sSTGPA}

As mentioned in the Introduction, we want to study gradings on the associative algebras of the form $\mathfrak{F}_\Pi(U)$ where $U$ is a vector space over $\FF$ and $\Pi\subset U^\ast$ is a total subspace. A natural class to which these algebras belong is the primitive algebras with minimal left ideals. In fact, the simple algebras in this class are precisely the algebras $\mathfrak{F}_\Pi(U)$ where $U$ is a right vector space over a division algebra and $\Pi\subset U^\ast$ is a total subspace. We are now going to develop a graded version of the theory in \cite[Chapter IV]{JSR}, which will apply to the setting we are interested in --- thanks to the following lemma. All algebras in this section will be associative, but not necessarily unital. Unless indicated otherwise, the ground field $\FF$ is arbitrary. (In fact, some results will also apply to rings.)

\begin{lemma}\label{lFPGGP} Let $\cR$ be a primitive algebra (or ring) with minimal left ideals. Suppose that $\cR$ is given a grading by a group $G$. Then $\cR$ is graded primitive with minimal graded left ideals.
\end{lemma}

\begin{proof}  
Any nonzero left ideal of $\cR$ is a faithful module, because $\cR$ is prime. Now consider $\cR$ as a graded left $\cR$-module. According to \cite[Proposition 2.7.3]{NVO}, the graded socle $S^\gr(\cR)$ always contains the ordinary socle $S(\cR)$. Hence $\cR$ must contain a minimal graded left ideal $I$. Since $I$ is a faithful graded simple $\cR$-module, $\cR$ is graded primitive. 
\end{proof}

\subsection{A structure theorem}

Let $\cR$ be a $G$-graded algebra (or ring) and let $V$ be a graded simple faithful left module over $\cR$. Let $\cD=\Endgr_\cR(V)$. We will follow the standard convention of writing the elements of $\cD$ on the right. By a version of Schur's Lemma (see e.g. \cite[Proposition 2.7.1]{NVO}), $\cD$ is a graded division algebra, and thus $V$ is a graded right vector space over $\cD$. Hence $\cR$ is isomorphic to a graded subalgebra of $\Endgr_\cD(V)$.

The {\em graded dual} $V^{\gr\ast}$ is defined as $\Homgr_\cD(V,\cD)$. Note that $V^{\gr\ast}$ is a graded left vector space over $\cD$, with the action of $\cD$ defined by $(df)(v)=df(v)$ for all $d\in\cD$, $v\in V$ and $f\in V^{\gr\ast}$. We will often use the symmetric notation $(f,v)$ for $f(v)$. 

Let $W\subset V^{\gr\ast}$ be a total graded subspace, i.e., the restriction of the mapping $(\,,\,)$ is a nondegenerate $\cD$-bilinear form $W\times V\to\cD$. Recall the graded subalgebras $\lvw$ and $\fvw$ of $\Endgr_\cD(V)$ defined in the Introduction: $\lvw$ is the set of all operators in $\Endgr_\cD(V)$ that are continuous with respect to the topology induced by $W$ and the subalgebra $\fvw\subset\lvw$ consists of all finite sums of operators of the form $v\ot w$ where $v\in V$ and $w\in W$ are homogeneous and $v\ot w$ acts on $V$ as follows:
\[
(v\ot w)(u)\bydef v(w,u)\quad\mbox{for all}\quad u,v\in V\;\mbox{ and }\;w\in W.
\]
It is easy to see that $V$ is a graded simple module for $\fvw$ and hence for $\lvw$ or any graded subalgebra of $\Endgr_\cD(V)$ containing $\fvw$.

Suppose that $\cR$ has a minimal graded left ideal $I$. Then either $I^2=0$ or $I=\cR\ve$ where $\ve$ is a homogeneous idempotent (hence of degree $e$). Indeed, if $I^2\ne 0$ then there is a homogeneous $x\in I$ such that $Ix\ne 0$ and hence $Ix=I$. The left annihilator $L$ of $x$ in $\cR$ is a graded left ideal of $\cR$ such that $I\not\subset L$ and hence $L\cap I=0$. Let $\ve$ be an element of $I$ such that $\ve x=x$. Since $x$ is homogeneous and every homogeneous component of $\ve$ is in $I$, we may assume that $\ve$ is homogeneous of degree $e$. Now $\ve^2 x=\ve x=x$ and hence $\ve^2-\ve\in L$. Since $I\cap L=0$, we conclude that $\ve^2=\ve$. Since $Ix\ne 0$, it follows that $I\ve\ne 0$ and hence $I\ve=I$. Since we are assuming $\cR$ graded primitive and hence graded prime, the case $I^2=0$ is not possible. Hence $I=\cR\ve$ is a graded simple $\cR$-module. 

The following is a graded version of a result in \cite[III.5]{JSR}.

\begin{lemma}\label{unique_graded_module}
Let $\cR$ be a graded primitive algebra (or ring) with a minimal graded left ideal $I$. Let $V$ be a faithful graded simple left $\cR$-module. Then there exists $g\in G$ such that $V$ is isomorphic to $I^{[g]}$ as a graded $\cR$-module. Hence, all faithful graded simple left $\cR$-modules are isomorphic up to a right shift of grading.
\end{lemma}

\begin{proof}
Since $IV$ is a graded submodule of $V$, we have either $IV=0$ or $IV=V$. But $V$ is faithful, so $IV=V$. Pick a homogeneous $v\in V$ such that $Iv\neq 0$ and let $g=\deg v$. Then the map $I\to V$ given by $r\mapsto rv$ is a homomorphism of $\cR$-modules and sends $I_h$ to $V_{hg}$, $h\in G$. By graded simplicity of $I$ and $V$, this map is an isomorphism of $\cR$-modules. Hence $I^{[g]}$ is isomorphic to $V$ as a graded $\cR$-module.
\end{proof}

Now we will obtain a graded version of the structure theorem in \cite[IV.9]{JSR}.

\begin{theorem}\label{theorem_graded_primitive} 
Let $\cR$ be a $G$-graded algebra (or ring). Then $\cR$ is graded primitive with minimal graded left ideals if and only if there exists a $G$-graded division algebra $\cD$, a graded right vector space $V$ over $\cD$ and a total graded subspace $W\subset V^{\gr\ast}$ such that $\cR$ is isomorphic to a graded subalgebra (subring) of $\lvw$ containing $\fvw$. Moreover, $\fvw$ is the only such subalgebra (subring) that is graded simple. 
\end{theorem}

\begin{proof}
Suppose $\cR$ is graded primitive and $I$ is a minimal graded left ideal of $\cR$. We already know that $I=\cR\veps$ for some homogeneous idempotent $\veps$ and that $I$ is a right vector space over the graded division algebra $\cD=\Endgr_\cR(I)$. Since $I$ is a faithful $\cR$-module, we can identify $\cR$ with the graded subalgebra of $\Endgr_\cD(I)$ consisting of the left multiplications $L_a\colon x\mapsto ax$ (restricted to $I$) for all $a\in\cR$. We claim that $\cD$ can be identified with $\veps\cR\veps$ and that 
\begin{equation}\label{eq:between_F_and_L}
\CF{I}{J}\subset\cR\subset\CL{I}{J}
\end{equation}
where $J\bydef\veps\cR$ is identified with a total graded subspace of $I^{\gr\ast}$ by virtue of the $\cD$-bilinear form $J\times I\to\cD$ given by $(y,x)=yx$ for all $x\in I$ and $y\in J$.

Indeed, for any nonzero homogeneous element $d\in\veps\cR\veps$ of degree $g\in G$, the right multiplication $R_d\colon x\mapsto xd$ yields  a nonzero endomorphism of the left $\cR$-module $I$, which is also homogeneous of degree $g$. Let $\vphi$ be an element of $\End_\cR(I)$. Clearly, $\vphi$ is completely determined by the image of $\veps$, which is $a\veps$, for some $a\in\cR$. Now $\veps\vphi=\veps^2\vphi=\veps a\veps$ and, for all $x\in I$, we have $x\vphi=(x\veps)\vphi=x(\veps a\veps)$. Hence, $\vphi$ coincides with the right multiplication by the element $\veps a\veps\in\veps\cR\veps$ (restricted to $I$), so we have proved that $d\mapsto R_d$ maps $\veps\cR\veps$ onto $\End_\cR(I)$. This allows us to identify $\Endgr_\cR(I)=\End_\cR(I)$ with the graded subalgebra $\veps\cR\veps\subset\cR$.

For any homogeneous $y\in J$, consider $f_y\colon I\to\cD$ given by $f_y(x)=yx$. Clearly, $f_y\in I^{\gr\ast}$. Since $I$ is faithful, the mapping $y\mapsto f_y$ is injective and hence identifies $J$ with a graded subspace of $I^{\gr\ast}$. Since $J$ is faithful (as a right $\cR$-module), $J$ is a total graded subspace of $I^{\gr\ast}$. It remains to verify \eqref{eq:between_F_and_L}. 

For any $a\in\cR$, the adjoint operator $L_a^\ast:I^{\gr\ast}\to I^{\gr\ast}$ is given by $(L_a^\ast(f))(x)=f(ax)$, for all $x\in I$ and $f\in I^{\gr\ast}$. If $f=f_y$ for some $y\in J$, then we have $(L_a^*(f_y))(x)=f_y(ax)=yax=f_{ya}(x)$. This shows that $L_a^\ast(f_y)=f_{ya}$, where $ya\in J$, so $L_a$ leaves the subspace $J\subset I^{\gr\ast}$ invariant. We have shown that $\cR\subset\CL{I}{J}$. Now we want to show that $\CF{I}{J}\subset \cR$. It suffices to verify that, for all $x\in I$ and $y\in J$, the operator $x\ot f_y$ is in $\cR$. For any $z\in I$, we have $(x\ot f_y)(z)=x f_y(z)=xyz=L_{xy}(z)$. So $x\ot f_y=L_{xy}\in\cR$, as required. The proof of the claim is complete.

Conversely, suppose $\cR$ is a graded algebra such that $\fvw\subset\cR\subset\lvw$ for some $G$-graded division algebra $\cD$, a graded right $\cD$-vector space $V$ and a total graded subspace $W\subset V^{\gr\ast}$. Then $V$ is a faithful graded simple module for $\cR$. To produce a minimal graded left ideal in $\cR$, take any nonzero homogeneous $f\in W$ and let $I$ be the set of all operators $v\otimes f$ with $v\in V$. Clearly, $I$ is graded and $\fvw\subset\cR$ implies $I\subset\cR$. Since, for any $r\in\cR$, we have $r(v\ot f)=r(v)\ot f$, we conclude that $I$ is a left ideal of $\cR$. Now, if $v\ot f$ is any nonzero homogeneous element of $I$, then $v$ is homogeneous and hence we can choose a homogeneous $f'\in W$ such that $f'(v)=1$. If $u\ot f\in I$ is another homogeneous element, then $(u\ot f')(v\ot f)=u\ot f$. Since $\fvw\subset\cR$, it follows that $\cR(v\ot f)=I$, proving that $I$ is a minimal graded ideal of $\cR$. 

Finally, it is easy to verify that $\fvw$ is graded simple and that it is a two-sided ideal of $\lvw$. Hence, $\cR$ as above will be graded simple if and only if $\cR=\fvw$.
\end{proof}

%We state the following graded version of a standard result (see e.g. \cite[IV.3]{JSR}) for future reference.

%\begin{lemma}\label{Re_and_eR}
%Let $\cR$ be a graded semiprime algebra (or ring) and let $\veps\in\cR$ be a homogeneous idempotent. 
%Then $\cR\veps$ is a minimal graded left ideal if and only if $\veps\cR$ is a minimal graded right ideal.
%\end{lemma} 

%\begin{proof}
%We will prove that $I\bydef\cR\veps$ is minimal if and only if $\cD\bydef\veps\cR\veps$ is a graded division algebra. 
%The result then follows by symmetry. If $I$ is minimal, then it is a graded simple left $\cR$-module, 
%so $\cD=\Endgr_\cR(I)$ is a graded division algebra. Conversely, assume that $\cD\bydef\veps\cR\veps$ is a graded division algebra. 
%Let $x\in I$ be a nonzero homogeneous element. Since $\cR$ is graded semiprime, we have $\cR x\ne 0$ and hence $(\cR x)^2\ne 0$. 
%So there exists a homogeneous $a\in\cR$ such that $xax\ne 0$. Since $x=x\veps$, it follows that $\veps ax\ne 0$. 
%But this is a homogeneous element of $\cD$, so there exists $b\in\cD$ such that $b\veps ax=\veps$. 
%Therefore, $\veps\in\cR x$ and hence $\cR x=I$.
%\end{proof}

\subsection{Isomorphisms}

We are going to investigate under what conditions two graded simple algebras described by Theorem \ref{theorem_graded_primitive} are isomorphic. By Lemma \ref{unique_graded_module}, $V$ is determined by $\cR$ up to isomorphism and shift of grading. If $\vphi\in\Endgr_\cR{V}$ is homogeneous of degree $t$, then $\vphi$ regarded as a map $V^{[g]}\to V^{[g]}$ will be homogeneous of degree $g^{-1}tg$. Indeed, $(V_h)\vphi\subset V_{ht}$, for all $h\in G$, can be rewritten as $(V_{hg}^{[g]})\vphi\subset V_{htg}^{[g]}$. Hence if $\cD=\Endgr_{\cR}(V)$, then $\Endgr_{\cR}(V^{[g]})={}^{[g^{-1}]}\cD^{[g]}$. It remains to include the total graded subspace $W\subset V^{\gr\ast}$ in the picture. It is convenient to introduce the following terminology. Fix a group $G$. 

\begin{df}
Let $\cD$ and $\cD'$ be $G$-graded division algebras (or rings) and let $V$ and $V'$ be graded right vector spaces over $\cD$ and $\cD'$, respectively. Let $\psi_0\colon\cD\to\cD'$ be an isomorphism of $G$-graded algebras (or rings). A homomorphism $\psi_1\colon V\to V'$ of $G$-graded vector spaces over $\FF$ (or $G$-graded abelian groups) is said to be {\em $\psi_0$-semilinear} if $\psi_1(vd)=\psi_1(v)\psi_0(d)$ for all $v\in V$ and $d\in\cD$. The {\em adjoint} to $\psi_1$ is the mapping $\psi_1^*:(V')^{\gr\ast}\to V^{\gr\ast}$ defined by $(\psi_1^*(f))(v)=\psi_0^{-1}(f(\psi_1(v)))$ for all $f\in(V')^{\gr\ast}$ and $v\in V$.  
\end{df}

One easily checks that $\psi_1^*$ is $\psi_0^{-1}$-semilinear.

\begin{df}
Let $\cD$ and $\cD'$ be $G$-graded division algebras (or rings), let $V$ and $V'$ be graded right vector spaces over $\cD$ and $\cD'$, respectively, and let $W$ and $W'$ be total graded subspaces of $V^{\gr\ast}$ and $(V')^{\gr\ast}$, respectively. An \emph{isomorphism of triples} from $(\cD,V,W)$ to $(\cD',V',W')$ is a triple $(\psi_0,\psi_1,\psi_2)$ where $\psi_0\colon\cD\to\cD'$ is an isomorphism of graded algebras (or rings) while $\psi_1\colon V\to V'$ and $\psi_2\colon W\to W'$ are isomorphisms of graded vector spaces over $\FF$ (or graded abelian groups) such that $(\psi_2(w),\psi_1(v))=\psi_0((w,v))$ for all $v\in V$ and $w\in W$.
\end{df}

It follows that $\psi_1$ and $\psi_2$ are $\psi_0$-semilinear. Also, for given isomorphisms $\psi_0$ and $\psi_1$ there can exist at most one $\psi_2$ such that $(\psi_0,\psi_1,\psi_2)$ is an isomorphism of triples. Such $\psi_2$ exists if and only if $\psi_1$ is $\psi_0$-semilinear and $\psi_1^*(W')= W$. Indeed, we can take $\psi_2$ to be the restriction of $(\psi_1^*)^{-1}$ to $W$. The condition $\psi_1^*(W')= W$ means that $\psi_1\colon V\to V'$ is a homeomorphism with respect to the topologies induced by $W$ and $W'$.

The following is a graded version of the isomorphism theorem in \cite[IV.11]{JSR}. 

\begin{theorem}\label{isomorphism_graded_simple} 
Let $G$ be a group. Let $\cD$ and $\cD'$ be $G$-graded division algebras (or rings), let $V$ and $V'$ be graded right vector spaces over $\cD$ and $\cD'$, respectively, and let $W$ and $W'$ be total graded subspaces of $V^{\gr\ast}$ and $(V')^{\gr\ast}$, respectively. Let $\cR$ and $\cR'$ be $G$-graded algebras (or rings) such that 
\[
\CF{V}{W}\subset\cR\subset\CL{V}{W}
\quad\mbox{and}\quad
\CF{V'}{W'}\subset\cR'\subset\CL{V'}{W'}.
\] 
If $\psi\colon\cR\to\cR'$ is an isomorphism of graded algebras, then there exist $g\in G$ and an isomorphism $(\psi_0,\psi_1,\psi_2)$ from $({}^{[g^{-1}]}\cD^{[g]},V^{[g]},{}^{[g^{-1}]}W)$ to $(\cD',V',W')$ such that 
\begin{equation}\label{eq:psi_of_r}
\psi(r)=\psi_1\circ r\circ\psi_1^{-1}\quad\mbox{for all}\quad r\in\cR.
\end{equation}
If other $g'\in G$ and isomorphism $(\psi'_0,\psi'_1,\psi'_2)$ from $({}^{[(g')^{-1}]}\cD^{[g']},V^{[g']},{}^{[(g')^{-1}]}W)$ to $(\cD',V',W')$ define $\psi$ as above, then there exists a nonzero homogeneous $d\in\cD$ such that $g'=g\deg d$, $\psi'_0(x)=d^{-1}\psi_0(x)d$ for all $x\in\cD$, $\psi'_1(v)=\psi_1(v)d$ for all $v\in V$, and $\psi'_2(w)=d^{-1}\psi_2(w)$ for all $w\in W$.

As a partial converse, if $(\psi_0,\psi_1,\psi_2)$ is an isomorphism of triples as above, then setting $\psi(r)=\psi_1\circ r\circ\psi_1^{-1}$ defines an isomorphism of $G$-graded algebras (or rings) $\psi\colon\CL{V}{W}\to\CL{V'}{W'}$ such that $\psi(\CF{V}{W})=\CF{V'}{W'}$.
\end{theorem}

\begin{proof}
Define a left $\cR$-module structure on $V'$ by setting $rv'\bydef\psi(r)v'$ for all $r\in\cR$ and $v'\in V'$. 
Then $V'$ is a faithful graded simple $\cR$-module and hence, by Lemma \ref{unique_graded_module}, there exists an isomorphism $\psi_1\colon V^{[g]}\to V'$ for some $g\in G$. By our definition of $\cR$-module structure on $V'$, we have $\psi_1(rv)=\psi(r)\psi_1(v)$ for all $r\in\cR$ and $v\in V$, i.e., \eqref{eq:psi_of_r} holds. Since ${}^{[g^{-1}]}\cD^{[g]}=\Endgr_{\cR}(V^{[g]})$ and $\cD'=\Endgr_{\cR}(V')$, we can define $\psi_0$ by setting $(v')(\psi_0(d))\bydef \psi_1(\psi_1^{-1}(v')d)$ for $v'\in V'$ and $d\in\cD$. Then $\psi_1$ is $\psi_0$-semilinear.

Now we need to show that $\psi_1^*(W')=W$. Consider a nonzero linear operator $r=v\ot w\in\CF{V}{W}$ where $v$ is homogeneous. We have $\psi (v\ot w)\in\CL{V'}{W'}$. Therefore, $(\psi (v\ot w))^*(W')\subset W'$. For any $v'\in V'$, we have
\begin{align*}
(\psi(v\ot w))(v')&=\psi_1((v\ot w)(\psi_1^{-1}(v')))=\psi_1(v(w,\psi_1^{-1}(v')))\\
&=\psi_1(v)\psi_0((w,\psi_1^{-1}(v'))).
\end{align*} 
Hence the $\cD'$-linear map $f\colon V'\to\cD'$ sending $v'\to\psi_0((w,\psi_1^{-1}(v')))$ is the composition of $\psi(v\ot w)$ and any element $f'\in W'$ that sends $\psi_1(v)$ to $1$. In other words, $f=(\psi (v\ot w))^*(f')$ and hence $f\in W'$. On the other hand,  $f=(\psi_1^{-1})^*(w)$ by definition. We have proved that $(\psi_1^{-1})^*(W)\subset W'$. By a similar argument applied to $\psi^{-1}$ instead of $\psi$, we obtain $\psi_1^*(W')\subset W$. Therefore, $\psi_1^*(W')=W$, as required. It follows that there exists $\psi_2\colon {}^{[g^{-1}]}W\to W'$ such that $(\psi_0,\psi_1,\psi_2)$ is an isomorphism of triples. 

If we have other $g'$ and $(\psi'_0,\psi'_1,\psi'_2)$, then $\psi'_1\circ\psi_1^{-1}$ is a homogeneous automorphism of the graded $\cR$-module $V'$, hence there exists a nonzero homogeneous $d\in\cD$ such that $\psi'_1(v)=\psi_1(v)d$ for all $v\in V$. Comparing degrees, we see that $g'=g\deg d$. Since $\psi'_1(vx)=\psi_1(v)\psi_0(x)d=\psi'(v)d^{-1}\psi_0(x)d$, we get $\psi'_0(x)=d^{-1}\psi_0(x)d$ for all $x\in\cD$. Substituting our expressions for $\psi'_0$ and $\psi'_1$ into the equation $(\psi'_2(w),\psi'_1(v))=\psi'_0((w,v))$, we get $(\psi'_2(w),\psi_1(v))=d^{-1}\psi_0((w,v))$ for all $v\in V$ and $w\in W$, which implies $d\psi'_2(w)=\psi_2(w)$.

Conversely, given an isomorphism of triples $(\psi_0,\psi_1,\psi_2)$, we know that $\psi_1$ is $\psi_0$-semilinear and $\psi_1^*(W')=W$. Let $r\in\Endgr_\cD(V)$ be homogeneous of degree $h\in G$. Consider the map $\psi_1\circ r\circ\psi_1^{-1}\colon V'\to V'$. For any $k\in G$, $\psi_1^{-1}$ sends $V'_k$ to $V_{kg^{-1}}$, then $r$ sends $V_{kg^{-1}}$ to $V_{hkg^{-1}}$, and, finally, $\psi_1$ sends $V_{hkg^{-1}}$ to $V'_{hk}$. Hence $\psi_1\circ r\circ\psi_1^{-1}$ is homogeneous of degree $h$. Since $\psi_1^{-1}$ is $\psi_0^{-1}$-semilinear and $\psi_1$ is $\psi_0$-semilinear, the composition $\psi_1\circ r\circ\psi_1^{-1}$ is a linear operator on $V'$. We have proved that the map $\psi\colon r\mapsto \psi_1\circ r\circ\psi_1^{-1}$ is an isomorphism of $G$-graded algebras $\Endgr_{\cD}(V)\to\Endgr_{\cD'}(V')$. Since $\psi_1^*(W')=W$, we have the following: $r\in\CL{V}{W}$ if and only if $r^*(W)\subset W$ if and only if $(\psi_1\circ r\circ\psi_1^{-1})^*(W')\subset W'$ if and only if $\psi(r)\in\CL{V'}{W'}$. Therefore, $\psi$ restricts to an isomorphism of graded algebras $\CL{V}{W}\to\CL{V'}{W'}$. It is easy to verify that 
\begin{equation}\label{eq:psi_on_CF}
\psi(v\ot w)=\psi_1(v)\ot\psi_2(w)\quad\mbox{for all}\quad v\in V\;\mbox{ and }\;w\in W.
\end{equation} 
It follows that $\psi(\CF{V}{W})=\CF{V'}{W'}$. 
\end{proof}

Note that it follows that any isomorphism of graded algebras $\psi:\cR\to\cR'$ extends to an isomorphism  $\CL{V}{W}\to\CL{V'}{W'}$ and restricts to an isomorphism $\CF{V}{W}\to\CF{V'}{W'}$.

\subsection{Graded simple algebras with minimal graded left ideals}\label{ssGSAMI}

Fix a grading group $G$. In view of Theorems \ref{theorem_graded_primitive} and \ref{isomorphism_graded_simple}, graded simple algebras (or rings) with minimal graded left ideals are classified by the triples $(\cD,V,W)$ where $\cD$ is a graded division algebra (or ring), $V$ is a right graded vector space over $\cD$ and $W$ is a total graded subspace of $V^{\gr\ast}$. For a fixed $\cD$, the triples $(\cD,V,W)$ can be classified up to isomorphism as follows. 

Let $T$ be the support of $\cD$ and let $\Delta=\cD_e$. Clearly, $T$ is a subgroup of $G$ and $\Delta$ is a division algebra (or ring). Consider the set $G/T$ of left cosets and the set $T\backslash G$ of right cosets of $T$ in $G$. The map $A\to A^{-1}$ is a bijection between $G/T$ and $T\backslash G$. Clearly, the graded right $\cD$-modules ${}^{[g]}\cD$ and ${}^{[h]}\cD$ are isomorphic if and only if $gT=hT$ (similarly for graded left $\cD$-modules). Any graded right $\cD$-module $V$ is a direct sum of modules of this form, which can be grouped into isotypic components. Namely, $V_A\bydef\bigoplus_{a\in A}V_a$ is the isotypic component of $V$ corresponding to ${}^{[g]}\cD$ where $A=gT$. Note that $V_g$ is a right $\Delta$-module and $V_A\cong V_g\ot_\Delta {}^{[g]}\cD$ as graded right $\cD$-modules. Select a left transversal $S$ for $T$, i.e., a set of representatives for the left cosets of $T$, and let $\tV_A=V_g$ where $g$ is the unique element of $A\cap S$. Let $\tV\bydef\bigoplus_{A\in G/T}\tV_A$. Then $\tV$ is a right $\Delta$-module and $V\cong\tV\ot_\Delta\cD$ as ungraded right $\cD$-modules. We can recover the $G$-grading on $V$ if we consider $\tV$ as graded by the set $G/T$. Similarly, any graded left $\cD$-module $W$ can be encoded by a left $\Delta$-module $\tW$ with a grading by the set $T\backslash G$. 

Now observe that since $T$ is the support of $\cD$, we have $(W_B,V_A)=0$ for all $A\in G/T$ and $B\in T\backslash G$ with $BA\ne T$. It follows that $W_{A^{-1}}$ is a total graded subspace of $(V_A)^{\gr\ast}$. Selecting a left transversal $S$ and using $S^{-1}$ as a right transversal, we obtain $\tV$ and $\tW$ such that $(\tW_{A^{-1}},\tV_A)\subset\Delta$ and 
\begin{equation}\label{eq:dual_stubs}
(\tW_{B},\tV_A)=0\quad\mbox{for all}\quad A\in G/T\;\mbox{ and }\;B\in T\backslash G\;\mbox{ with }\;B\ne A^{-1}. 
\end{equation}
Hence, for each $A\in G/T$, the $\cD$-bilinear form $W\times V\to\cD$ restricts to a nondegenerate $\Delta$-bilinear form $\tW_{A^{-1}}\times\tV_A\to\Delta$, which identifies $\tW_{A^{-1}}$ with a total subspace of the $\Delta$-dual $(\tV_A)^\ast$ .

Conversely, let $\tV$ be a right $\Delta$-module that is given a grading by $G/T$ and let $\tW$ be a left $\Delta$-module that is given a grading by $T\backslash G$. Suppose $\tW_{A^{-1}}$ is identified with a total subspace of $(\tV_A)^\ast$ for each $A\in G/T$ or, equivalently, we have a nondegenerate $\Delta$-bilinear form $\tW\times\tV\to\Delta$ that satisfies \eqref{eq:dual_stubs}. For each $A\in G/T$, choose $g\in A$ and let $V_A=\tV_A\ot_\Delta{}^{[g]}\cD$. Then $V_A$ is a graded right $\cD$-module whose isomorphism class does not depend on the choice of $g$. Also, using the same $g$, let $W_{A^{-1}}=\cD^{[g]}\ot_\Delta\tW_{A^{-1}}$. Set $V=\bigoplus_{A\in G/T}V_A$ and $W=\bigoplus_{B\in T\backslash G}W_B$. Extending the $\Delta$-bilinear form $\tW\times\tV\to\Delta$ to a $\cD$-bilinear form $W\times V\to\cD$, we can identify $W$ with a total graded subspace of $V^{\gr\ast}$. We will denote the corresponding $G$-graded algebra $\fvw$ by $\frF(G,\cD,\tV,\tW)$.

\begin{df}\label{df:datum_finitary_general}
We will write $(\cD,\tV,\tW)\sim (\cD',\tV',\tW')$ if there is an element $g\in G$ and an isomorphism $\psi_0\colon {}^{[g^{-1}]}\cD^{[g]}\to\cD'$ of graded algebras such that, for any $A\in G/T$, there exists an isomorphism of triples from $(\Delta,\tV_{A},\tW_{A^{-1}})$ to $(\Delta',\tV'_{Ag},\tW'_{g^{-1}A^{-1}})$ whose component $\Delta\to\Delta'$ is the restriction of $\psi_0$. (Note that $Ag$ is a left coset for $T'=g^{-1}Tg$.) 
\end{df}

\begin{corollary}\label{classification_iso_assoc_graded_simple}
Let $G$ be a group and let $\cR$ be a $G$-graded algebra (or ring). If $\cR$ is graded simple with minimal graded left ideals, then $\cR$ is isomorphic to some $\frF(G,\cD,\tV,\tW)$. Two graded algebras $\frF(G,\cD,\tV,\tW)$ and $\frF(G,\cD',\tV',\tW')$ are isomorphic if and only if $(\cD,\tV,\tW)\sim (\cD',\tV',\tW')$.
\end{corollary}

\begin{proof}
The first claim is clear by Theorem \ref{theorem_graded_primitive} and the above discussion. Definition \ref{df:datum_finitary} is set up in such a way that $(\cD,\tV,\tW)\sim (\cD',\tV',\tW')$ if and only if the triples $({}^{[g^{-1}]}\cD^{[g]},V^{[g]},{}^{[g^{-1}]}W)$ and $(\cD',V',W')$ are isomorphic for some $g\in G$, so the second claim follows from Theorem \ref{isomorphism_graded_simple}.
\end{proof}
 
An important special case of Corollary \ref{classification_iso_assoc_graded_simple} is where the graded simple algebra $\cR$ satisfies the descending chain condition on graded left ideals. Then $V$ is finite-dimensional over $\cD$, so $W=V^{\gr\ast}=V^\ast$ and $\cR$ is isomorphic to the matrix algebra $M_k(\cD)$ where $k=\dim_\cD V$. Moreover, $V=V_{A_1}\oplus\cdots\oplus V_{A_s}$ for some distinct  $A_1,\ldots,A_s\in G/T$, which can be encoded by two $s$-tuples: $\kappa=(k_1,\ldots,k_s)$ and $\gamma=(g_1,\ldots,g_s)$ where $k_i\bydef\dim_\cD V_{A_i}$ are positive integers with $k_1+\cdots+k_s=k$ and $g_i$ are representatives for the cosets $A_i$. Therefore, the said algebras $\cR$ are classified by the triples $(\cD,\kappa,\gamma)$, up to an appropriate equivalence relation. Explicitly, the grading on $\cR$ can be written as follows. If $\{v_1,\ld,v_k\}$ is a homogeneous basis of $V$ over $\cD$, with $\deg v_i=h_i$, and $\{v^1,\ld,v^k\}$ is the dual basis of $V^\ast$, then $\deg v^i=h_i^{-1}$ and, for any homogeneous $d\in\cD$, the degree of the operator $v_i d\otimes v^j=v_i\otimes d v^j$ (which is represented by the matrix with $d$ in position $(i,j)$ and zeros elsewhere) equals $h_i(\deg d)h_j^{-1}$. This classification (under the assumption that $\cR$ is finite-dimensional over $\FF$) already appeared in the literature --- see \cite{BK} and references therein. The $G$-gradings on $M_k(\cD)$ defined in this way by $k$-tuples $(h_1,\ldots,h_k)$ of elements in $G$ will be called {\em elementary}.

In general, $\frF(G,\cD,\tV,\tW)$ can be written as $\tV\ot_\Delta\cD\ot_\Delta\tW$ where, for any $\tilde v\in\tV$, $d\in\cD$, $\tilde w\in\tW$, the element $\tilde v\ot d\ot\tilde w$ acts on $V=\tV\ot_\Delta\cD$ as follows: $(\tilde v\ot d\ot\tilde w)(u\ot a)=\tilde v\ot d(w,u)a$ for all $\tilde u\in\tV$ and $a\in\cD$. Clearly, the multiplication of $\frF(G,\cD,\tV,\tW)$ is given by 
\begin{equation}\label{eq:mult_frF}
(\tilde v\ot d\ot\tilde w)(\tilde v'\ot d'\ot\tilde w')=\tilde v\ot d(\tilde w,\tilde v')d'\ot\tilde w'.
\end{equation}
Fixing a left transversal $S$ for $T$, the $G$-grading on $\frF(G,\cD,\tV,\tW)$ is given by 
\begin{equation}\label{eq:grad_frF}
\deg(\tilde v\ot d\ot\tilde w)=\gamma(A)t\gamma(B)^{-1}\quad\mbox{for all}\quad \tilde v\in\tV_A,\,d\in\cD_t,\,\tilde w\in\tW_{B^{-1}},
\end{equation}
where $t\in T$, $A,B\in G/T$, and $\gamma(A)$ denotes the unique element of $A\cap S$. The isomorphism class of the grading does not depend on the choice of the transversal.

It is known that, for any finite-dimensional subspaces $\tV_1\subset\tV$ and $\tW_1\subset\tW$, there exist finite-dimensional subspaces $\tV_0\subset \tV$ and $\tW_0\subset\tW$ such that $\tV_1\subset\tV_0$, $\tW_1\subset\tW_0$, and the restriction of the bilinear form $\tW\times\tV\to\Delta$ to $\tW_0\times \tV_0$ is nondegenerate (see e.g. \cite[Lemma 5.7]{Bar}). Selecting dual bases in $\tV_0$ and $\tW_0$, we see that $\tV_0\ot_\Delta\cD\ot_\Delta\tW_0$ is a subalgebra of $\frF(G,T,\tV,\tW)$ isomorphic to $M_k(\cD)$ where $k=\dim_\Delta\tV_0=\dim_\Delta\tW_0$. Without loss of generality, we may assume that $\tV_0$ is a graded subspace of $\tV$ with respect to the grading by $G/T$ and $\tW_0$ is a graded subspace of $\tW$ with respect to the grading by $T\backslash G$. Then our subalgebra $\tV_0\ot_\Delta\cD\ot_\Delta\tW_0$ is graded. Moreover, in terms of the matrix algebra $M_k(\cD)$, this grading is elementary.  Thus we obtain the following  graded version of Litoff's Theorem \cite[IV.15]{JSR} (cf. Theorem 4 in \cite{BZ10}):

\begin{corollary}\label{graded_Litoff}
Let $G$ be a group and let $\cR$ be a $G$-graded algebra (or ring). If $\cR$ is graded simple with minimal graded left ideals, then there exists a graded division algebra $\cD$ such that $\cR$ is a direct limit of matrix algebras over $\cD$ with elementary gradings.\qed
\end{corollary}

\subsection{Classification of $G$-gradings on the algebras $\frF_\Pi(U)$}

In this work, we are primarily interested in the case $\cR=\frF_\Pi(U)$ where $U$ is a vector space over $\FF$ and $\Pi$ is a total subspace of $U^\ast$. We will assume that $\FF$ is {\em algebraically closed}. Then the algebras of the form $\frF_\Pi(U)$ have the following abstract characterization: they are precisely the locally finite simple algebras with minimal left ideals. Indeed, $\frF_\Pi(U)=U\otimes\Pi$ is a direct limit of matrix algebras over $\FF$ and hence is simple and locally finite. Conversely, if $\cR$ is a locally finite simple algebra with minimal left ideals, then $\cR$ is isomorphic to $\frF_\Pi(U)$ where $U$ is a right vector space over a division algebra $\Delta$ and $\Pi$ is a total subspace of $U^\ast$. But $\Delta$ is isomorphic to a subalgebra of $\cR$, hence algebraic over $\FF$. Since $\FF$ is algebraically closed, this implies $\Delta=\FF$.

If $\cR$ is given a $G$-grading, then $\cR$ is graded simple with minimal graded left ideals (Lemma \ref{lFPGGP}), so we can apply Corollary \ref{classification_iso_assoc_graded_simple}. Hence $\cR$ is isomorphic to some $\frF(G,\cD,\tV,\tW)$ as a graded algebra. We claim that, disregarding the grading, $\cD$ is isomorphic to $M_\ell(\FF)$ for some $\ell$. 

Recall from the proof of Theorem \ref{theorem_graded_primitive} that we can represent $\cR$ as $\CF{I}{J}$ where $I=\cR\veps$ is a minimal graded left ideal, $\veps$ is a homogeneous idempotent, and $J=\veps\cR$. Recall also that $\cD=\Endgr_\cR(I)$ coincides with $\End_\cR(I)$ and is isomorphic to $\veps\cR\veps$. It is known that $\cR$ is semisimple as a left or right $\cR$-module (see e.g. \cite[IV.9]{JSR}). In fact, it is easy to see that, if $\cR$ is represented as $\frF_\Pi(U)$, then the mapping $U_0\mapsto U_0\ot\Pi$ is a one-to-one correspondence between the subspaces of $U$ and the right ideals of $\cR$ whereas the mapping $\Pi_0\mapsto U\ot\Pi_0$ is a one-to-one correspondence between the subspaces of $\Pi$ and the left ideals of $\cR$. Hence we can write $I=\cR\veps_1\oplus\cdots\oplus\cR\veps_\ell$ where $\veps_i$ are orthogonal idempotents with $\veps_1+\cdots+\veps_\ell=\veps$ and $\cR\veps_i$ are minimal (ungraded) left ideals. Each of the $\cR\veps_i$ is isomorphic to $U$ as a left $\cR$-module. Since $\End_\cR(U)=\FF$, it follows that the algebra $\End_\cR(I)$ is isomorphic to $M_\ell(\FF)$, completing the proof of the claim. 

If $\cR$ is represented as $\frF_\Pi(U)$, we can construct this isomorphism explicitly. Namely, write $I=U\ot\Pi_0$ where $\Pi_0\subset\Pi$ is an $\ell$-dimensional subspace and select $U_0\subset U$ such that the restriction of the bilinear form $\Pi\times U\to\FF$ to $\Pi_0\times U_0$ is nondegenerate. Let $\{e_1,\ldots,e_\ell\}$ be a basis of $U_0$ and let $\{e^1,\ldots,e^\ell\}$ be the dual basis of $\Pi_0$. Then we can take $\veps_i=e_i\ot e^i$, so $\cR\veps_i=U\ot\FF e^i$, and the elements $e_i\ot e^j$ constitute a basis of matrix units for $\veps\cR\veps$. 

So, $\cD$ is a matrix algebra which is given a $G$-grading that makes it a graded division algebra. Such gradings are well known in the literature under the names of {\em division gradings}, {\em fine gradings} (because all homogeneous components are $1$-dimensional) and, in the case of abelian $G$, also {\em Pauli gradings} (because generalized Pauli matrices can be used to construct them). If we choose, for each $t\in T$, a nonzero homogeneous element $X_t\in\cD$ of degree $t$, then $\cD_t=\FF X_t$ and hence $X_t X_{t'}=\sigma(t,t')X_{tt'}$ for some $2$-cocycle $\sigma\colon T\times T\to\FF^\times$. This shows that the graded algebra $\cD$ is isomorphic to a {\em twisted group algebra} $\FF^\sigma T$ (with its natural $T$-grading regarded as a $G$-grading). In the case of abelian $G$, division gradings on $M_\ell(\FF)$ by $G$ are classified up to isomorphism in \cite{BK}. Namely, the isomorphism classes of such gradings are in bijection with the pairs $(T,\beta)$ where $T\subset G$ is a subgroup of order $\ell^2$ and $\beta\colon T\times T\to\FF^\times$ is a nondegenerate alternating bicharacter. Here $T$ is the support of the grading and $\beta$ is given by $\beta(t,t')=\sigma(t,t')/\sigma(t',t)$, so we get 
\[
X_t X_{t'}=\beta(t,t')X_{t'} X_t\quad\mbox{for all}\quad t,t'\in T.
\]
Note that $\chr{\FF}$ cannot divide the order of $T$.

We want to understand the relation between, on the one hand, $\tV$ and $\tW$ and, on the other hand, $U$ and $\Pi$. We will use $I$ as $V$ and $J$ as $W$. The mapping $u\mapsto u\ot e^i$ is an isomorphism of left $\cR$-modules $U\to\cR\veps_i$. Also, the mapping $f\mapsto e_i\ot f$ is an isomorphism of right  $\cR$-modules $\Pi\to\veps_i\cR$. This allows us to identify $I$ with $U^\ell$ and $J$ with $\Pi^\ell$. Recall that the $\cD$-bilinear form $J\times I\to\cD$ is just the multiplication of $\cR$. Hence, under the above identifications, this  $\cD$-bilinear form maps $(f_1,\ldots,f_\ell)\in\Pi^\ell$ and $(u_1,\ldots,u_\ell)\in U^\ell$ to the matrix $[(f_i,u_j)]_{i,j}$ in $\cD$. Let $M$ be the unique simple right $\cD$-module and let $N$ be the unique simple left $\cD$-module, i.e., $M$ is $\FF^\ell$ written as rows and $N$ is $\FF^\ell$ written as columns. Then, disregarding the $G$-gradings on $I$ and $J$, we can identify $I$ with $U\ot M$ as an  $(\cR,\cD)$-bimodule and also identify $J$ with $N\ot\Pi$ as a $(\cD,\cR)$-bimodule. Under these identifications, the $\cD$-bilinear form $J\times I\to\cD$ coincides with the extension of the $\FF$-linear form $\Pi\times U\to\FF$. Now, we have $U\cong I\ot_\cD N$ and $\Pi\cong M\ot_\cD J$ as $\cR$-modules. If we identify $U$ with $I\ot_\cD N$ and $\Pi$ with $M\ot_\cD J$, then the $\FF$-bilinear form $\Pi\ot U\to\FF$ is related to the $\cD$-bilinear form $J\times I\to\cD$ by the following formula:
\begin{equation}\label{eq:myxn}
(m\ot y,x\ot n)=m(y,x)n\quad\mbox{for all}\quad m\in M, n\in N, x\in I, y\in J,
\end{equation}
where the right-hand side is the scalar in $\FF$ obtained by multiplying a row, a matrix and a column.  
Recall that $\tV$ and $\tW$ associated to $V$ and $W$ are defined in such a way that $V=\tV\ot\cD$ and $W=\cD\ot\tW$ as ungraded $\cD$-modules, and the $\cD$-bilinear form $W\times V\to\cD$ is the extension of the $\FF$-bilinear form $\tW\times\tV\to\FF$ (recall that $\cD_e=\FF$). Hence $U=V\ot_\cD N=(\tV\ot\cD)\ot_\cD N\cong\tV\ot N$ and $\Pi=M\ot_\cD W=M\ot_\cD(\cD\ot\tW)\cong M\ot\tW$ as vector spaces over $\FF$, where the isomorphism $(\tV\ot\cD)\ot_\cD N\cong\tV\ot N$ is given by $\tilde v\ot d\ot n\mapsto \tilde v\ot dn$ and the isomorphism $M\ot_\cD(\cD\ot\tW)\to M\ot\tW$ is given by $m\ot d\ot \tilde w\mapsto md\ot \tilde w$. Substituting $x=\tilde v\ot a$ and $y=b\ot \tilde w$, for any $\tilde v\in\tV$, $\tilde w\in\tW$, $a,b\in\cD$, into \eqref{eq:myxn}, we obtain 
\[
(m\ot b\ot \tilde w,\tilde v\ot a\ot n)=m(b\ot \tilde w,\tilde v\ot a)n=mb(\tilde w,\tilde v)an.
\] 
Hence, if we identify $U$ with $\tV\ot N$ and $\Pi$ with $M\ot\tW$, then the $\FF$-bilinear forms $\Pi\times U\to\FF$ and $\tW\times\tV\to \FF$ are related by the following formula:
\begin{equation}\label{eq:pairing_U_Pi}
(m\ot\tilde w,\tilde v\ot n)=(\tilde w,\tilde v)mn\quad\mbox{for all}\quad m\in M, n\in N, \tilde v\in \tV, \tilde w\in \tW.
\end{equation}
In other words, we can identify $U$ with $\tV^\ell$ and $\Pi$ with $\tW^\ell$ so that the above $\FF$-bilinear forms are related as follows: 
\begin{equation}\label{eq:two_bilinear_forms}
((\tilde w_1,\ldots,\tilde w_\ell),(\tilde v_1,\ldots,\tilde v_\ell))=\sum_{i=1}^\ell (\tilde w_i,\tilde v_i)\quad\mbox{for all}\quad \tilde v_i\in\tV\;\mbox{ and }\;\tilde w_i\in\tW.
\end{equation}

Finally, we observe that Definition \ref{df:datum_finitary_general} simplifies because $\cD_e=\cD'_e=\FF$. For brevity, an isomorphism of triples from $(\FF,U,\Pi)$ to $(\FF,U',\Pi')$ will be called an {\em isomorphism of pairs} from $(U,\Pi)$ to $(U',\Pi')$. In other words, pairs $(U,\Pi)$ and $(U',\Pi')$ are isomorphic if and only if there exists an isomorphism $U\to U'$ of vector spaces (over $\FF$) whose adjoint $(U')^\ast\to U^\ast$ maps $\Pi'$ onto $\Pi$. 

\begin{df}\label{df:datum_finitary}
We will write $(\cD,\tV,\tW)\sim (\cD',\tV',\tW')$ if there is an element $g\in G$ such that ${}^{[g^{-1}]}\cD^{[g]}\cong\cD'$ as graded algebras and, for any $A\in G/T$, we have $(\tV_{A},\tW_{A^{-1}})\cong(\tV'_{Ag},\tW'_{g^{-1}A^{-1}})$. 
\end{df}

To summarize:

\begin{theorem}\label{lfd}
Let $\cR$ be a locally finite simple algebra with minimal left ideals over an algebraically closed field $\FF$. If $\cR$ is given a grading by a group $G$, then $\cR$ is isomorphic to some $\frF(G,\cD,\tV,\tW)$ where $\cD$ is a matrix algebra over $\FF$ equipped with a division grading. Conversely, if $\cD=M_\ell(\FF)$ with a division grading, then $\frF(G,\cD,\tV,\tW)$ is a locally finite simple algebra with minimal left ideals, which can be represented as $\frF_\Pi(U)$ where $U=\tV^\ell$, $\Pi=\tW^\ell$, and the nondegenerate bilinear form $\Pi\times U\to\FF$ is given by \eqref{eq:two_bilinear_forms}. Moreover, two such graded algebras $\frF(G,\cD,\tV,\tW)$ and $\frF(G,\cD',\tV',\tW')$ are isomorphic if and only if $(\cD,\tV,\tW)\sim (\cD',\tV',\tW')$ in the sense of Definition \ref{df:datum_finitary}.\qed
\end{theorem}

If $G$ is abelian, then ${}^{[g^{-1}]}\cD^{[g]}=\cD$ and the isomorphism class of $\cD$ is determined by $(T,\beta)$. Hence we may write $\frF(G,T,\beta,\tV,\tW)$ for $\frF(G,\cD,\tV,\tW)$.

\begin{df}\label{df:datum_finitary_ab}
We will write $(\tV,\tW)\sim (\tV',\tW')$ if there is an element $g\in G$ such that, for any $A\in G/T$, we have $(\tV_{A},\tW_{A^{-1}})\cong(\tV'_{Ag},\tW'_{g^{-1}A^{-1}})$.
\end{df}

\begin{corollary}\label{lfd_abelian_G}
If in Theorem \ref{lfd} the group $G$ is abelian, then $\cR$ is isomorphic to some $\frF(G,T,\beta,\tV,\tW)$. Two such graded algebras $\frF(G,T,\beta,\tV,\tW)$ and $\frF(G,T',\beta',\tV',\tW')$ are isomorphic if and only if $T=T'$, $\beta=\beta'$ and $(\tV,\tW)\sim (\tV',\tW')$ in the sense of Definition \ref{df:datum_finitary_ab}.\qed 
\end{corollary}

In the case where $\cR$ has countable dimension, we can classify $G$-gradings on $\cR$ in combinatorial terms. Clearly, $\frF_\Pi(U)$ has countable dimension if and only if both $\Pi$ and $U$ have countable dimension. It is known that then there exist dual bases in $U$ and $\Pi$, hence all such pairs $(U,\Pi)$ are isomorphic and there is only one such algebra $\cR$, which is denoted by $M_\infty(\FF)$. We will state the classification of $G$-gradings in a form that also applies to $M_n(\FF)$, for which this result is known under the assumption that $G$ is abelian \cite{BK}. Up to isomorphism, the pairs $(\tV_{A},\tW_{A^{-1}})$ can be encoded by the function $\kappa\colon G/T\to\{0,1,2,\ldots,\infty\}$ that sends $A$ to $\dim\tV_A$. Note that the support of the function $\kappa$ is finite or countable, so $|\kappa|\bydef\sum_{A\in G/T}\kappa(A)$ is defined as an element of $\{0,1,2,\ldots,\infty\}$. We will denote the associated graded algebra $\frF(G,\cD,\tV,\tW)$ by $\frF(G,\cD,\kappa)$. Finally, for any $g\in G$, define $\kappa^g\colon G/(g^{-1}Tg)\to\{0,1,2,\ldots,\infty\}$ by setting $\kappa^g(Ag)\bydef\kappa(A)$ for all $A\in G/T$.

\begin{corollary}\label{gradings_on_M_infinity}
Let $\FF$ be an algebraically closed field and let $\cR=M_n(\FF)$ where $n\in\NN\cup\{\infty\}$. If $\cR$ is given a grading by a group $G$, then $\cR$ is isomorphic to some $\frF(G,\cD,\kappa)$ where $\cD=M_\ell(\FF)$, with $\ell\in\NN$ and $n=|\kappa|\ell$, is equipped with a division grading. Moreover, two such graded algebras $\frF(G,\cD,\kappa)$ and $\frF(G,\cD',\kappa')$ are isomorphic if and only if there exists $g\in G$ such that ${}^{[g^{-1}]}\cD^{[g]}\cong\cD'$ as graded algebras and $\kappa^g=\kappa'$.\qed
\end{corollary}

The graded algebra $\frF(G,\cD,\kappa)$ can be constructed explicitly as follows. 
Select a left transversal $S$ for $T$ and let $\gamma(A)$ be the unique element of $A\cap S$. For each $A\in G/T$, select dual bases $\{v_i(A)\}$ and $\{v^j(A)\}$ for $\tV_A$ and $\tW_A$, respectively, consisting each of $\kappa(A)$ vectors. Then the algebra $\frF(G,\cD,\kappa)$ has a basis 
\[
\{E^{A,B}_{i,j}(t)\;|\;A,B\in G/T,\,i,j\in\NN,i\le\kappa(A),j\le\kappa(B),\,t\in T\},
\] 
where $E^{A,B}_{i,j}(t)=v_i(A)X_t\ot v^j(B)=v_i(A)\ot X_t v^j(B)$. Equations \eqref{eq:mult_frF} and \eqref{eq:grad_frF} imply that the multiplication is given by the formula (using Kronecker delta):
\[
E^{A,B}_{i,j}(t)E^{A',B'}_{i',j'}(t')=\delta_{B,A'}\delta_{j,i'}\sigma(t,t')E^{A,B'}_{i,j'}(tt')
\]
and the $G$-grading is given by 
\[
\deg E^{A,B}_{i,j}(t)=\gamma(A)t\gamma(B)^{-1}.
\]
Thus we recover Theorem 5 in \cite{BZ10}, which asserts the existence of such a basis under the assumption that $G$ is a finite abelian group and $\FF$ is algebraically closed of characteristic zero. Theorem 6 in the same paper gives our condition for isomorphism of two $G$-gradings in the special case $\cD=\FF$.

\subsection{Antiautomorphisms and sesquilinear forms}\label{ssASF}

We want to investigate under what conditions a graded algebra described by Theorem \ref{theorem_graded_primitive} admits an antiautomorphism. So, we temporarily return to the general setting: $\cR$ is a $G$-graded primitive algebra (or ring) with minimal graded left ideals. 

We may assume $\CF{V}{W}\subset\cR\subset\CL{V}{W}$ where $V$ is a right vector space over a graded division algebra $\cD$, $W$ is a left vector space over $\cD$, and $W$ is identified with a total graded subspace of $V^{\gr\ast}$ by virtue of a $\cD$-bilinear form $(\,,\,)$. Thus, we have 
\begin{equation}\label{eq:D_bilinear}
(dw,v)=d(w,v)\quad\mbox{and}\quad (w,vd)=(w,v)d\quad\mbox{for all}\quad v\in V,\,w\in W,\,d\in\cD.
\end{equation}
Note that, since the adjoint of any operator $r\in\cR$ leaves $W$ invariant, $W$ becomes a graded right $\cR$-module such that 
\begin{equation}\label{eq:R_associative}
(wr,v)=(w,rv)\quad\mbox{for all}\quad v\in V,\,w\in W,\,r\in\cR.
\end{equation}
It follows that $\CF{W}{V}\subset\cR^{\mathrm{op}}\subset\CL{W}{V}$ where $W$ is regarded as a right vector space over $\cD^{\mathrm{op}}$, $V$ as a left vector space over $\cD^{\mathrm{op}}$, and $V$ is identified with a total graded subspace of $W^{\gr\ast}$ by virtue of $(v,w)^{\mathrm{op}}\bydef(w,v)$. (The gradings on all these objects are by the group $G^{\mathrm{op}}$.)

Now suppose that we have an antiautomorphism $\vphi$ of the graded algebra $\cR$. Since $\cS\bydef\fvw$ is the unique minimal graded two-sided ideal of $\cR$, we have $\vphi(\cS)=\cS$. Thus $\vphi$ restricts to an antiautomorphism of the graded simple algebra $\cS$. It is known (see \cite{BShZ}) that if a graded simple algebra admits an antiautomorphism, then the support of the grading generates an abelian group. Now observe that any element of the support of $\Endgr_\cD(V)$ has the form $gh^{-1}$ where $g$ and $h$ are in the support of $V$, and all elements of this form already occur in the support of $\fvw$. Hence, the support of $\cR$ equals the support of $\cS$ and generates an abelian group. For this reason, we will assume from now on that $G$ is {\em abelian}. 

Applying Theorem \ref{isomorphism_graded_simple} to the isomorphism $\vphi\colon\cR\to\cR^{\mathrm{op}}$ and taking into account that $G$ is abelian, we see that there exist $g_0\in G$ and an isomorphism $(\vphi_0,\vphi_1,\vphi_2)$ from $(\cD,V^{[g_0]},W^{[g_0^{-1}]})$ to $(\cD^{\mathrm{op}},W,V)$ such that $\vphi(r)=\vphi_1\circ r\circ\vphi_1^{-1}$. In particular, $\vphi_0$ is an antiautomorphism of the graded algebra $\cD$ and $\vphi_1$ is $\vphi_0$-semilinear:
\begin{equation}\label{eq:phi0_phi1}
\vphi_1(vd)=\vphi_0(d)\vphi(v)\quad\mbox{for all}\quad v\in V\;\mbox{ and }\;d\in\cD.
\end{equation}
Now define a nondegenerate $\FF$-bilinear form $B\colon V\times V\to\cD$ as follows:
\[
B(u,v)\bydef(\vphi_1(u),v)\quad\mbox{for all}\quad u,v\in V.
\]
Then $B$ has degree $g_0$ when regarded as a map $V\ot V\to\cD$. Combining \eqref{eq:D_bilinear} and \eqref{eq:phi0_phi1}, we see that, over $\cD$, the form $B$ is linear in the second argument and $\vphi_0$-semilinear in the first argument, i.e., 
\begin{equation}\label{eq:sesquilinear}
B(ud,v)=\vphi_0(d)B(u,v)\quad\mbox{and}\quad B(u,vd)=B(u,v)d\quad\mbox{for all}\quad u,v\in V,\,d\in\cD.
\end{equation} 
For brevity, we will say that $B$ is {\em $\vphi_0$-sesquilinear}. 

Applying \eqref{eq:R_associative}, we obtain for all $u,v\in V$ and $r\in\cR$:
\begin{align*}
B(ru,v)&=(\vphi_1(ru),v)=(\vphi_1(u)\vphi(r),v)\\
&=(\vphi_1(u),\vphi(r)v)=B(u,\vphi(r)v),
\end{align*}
which means that $\vphi(r)$ is {\em adjoint to $r$ with respect to $B$}. In particular, $\vphi$ can be recovered from $B$. 

We will need one further property of $B$. Consider 
\[
\bar{B}(u,v)\bydef\vphi_0^{-1}(B(v,u)). 
\]
Then $\bar{B}$ is a nondegenerate $\vphi_0^{-1}$-sesquilinear form of the same degree as $B$. Clearly, we have $\bar{B}(ru,v)=\bar{B}(u,\vphi^{-1}(r)v)$, so $\bar{B}$ is related to $\vphi^{-1}$ in the same way as $B$ is related to $\vphi$. We claim that there exists a $\vphi_0^{-2}$-semilinear isomorphism of graded vector spaces $Q\colon V\to V$ such that
\begin{equation}\label{eq:operator_Q}
\bar{B}(u,v)=B(Qu,v)\quad\mbox{for all}\quad u,v\in V.
\end{equation}
Indeed, $Q=\vphi_1^{-1}\circ\vphi_2^{-1}$ satisfies the requirements: it is clearly an invertible $\vphi_0^{-2}$-semilinear map, homogeneous of degree $e$, and we have
\[
B(Qu,v)=(\vphi_2^{-1}(u),v)=(v,\vphi_2^{-1}(u))^{\mathrm{op}}
=\vphi_0^{-1}((\vphi_1(v),u))=\bar{B}(u,v).
\]
It is important to note that the adjoint $Q^*=(\vphi_2^{-1})^*\circ(\vphi_1^{-1})^*=\vphi_1\circ\vphi_2$ maps $W$ onto $W$, i.e., $Q$ is a homeomorphism.

\begin{df}\label{weakly_Hermitian}
We will say that a nondegenerate homogeneous $\vphi_0$-sesquilinear form $B\colon V\times V\to\cD$ is {\em weakly Hermitian} if there exists a $\vphi_0^{-2}$-semilinear isomorphism $Q\colon V\to V$ of graded vector spaces such that \eqref{eq:operator_Q} holds. (Note that, since $B$ is nondegenerate, \eqref{eq:operator_Q} uniquely determines $Q$.)
\end{df}

The following is a graded version of the main result in \cite[IV.12]{JSR}.

\begin{theorem}\label{antiauto_graded_simple}
Let $G$ be an abelian group. Let $\cD$ be a $G$-graded division algebra (or ring), let $V$ be a graded right vector space over $\cD$, and let $W$ be a total graded subspace of $V^{\gr\ast}$. Let $\cR$ be a $G$-graded algebra (or ring) such that 
\[
\fvw\subset\cR\subset\lvw.
\]
If $\vphi$ is an antiautomorphism of the graded algebra $\cR$, then there exist an antiautomorphism $\vphi_0$ of the graded algebra $\cD$ and a weakly Hermitian nondegenerate homogeneous $\vphi_0$-sesquilinear form $B:V\times V\to\cD$, such that the following conditions hold:
\begin{enumerate}
\item[(a)] the mapping $V\to V^{\gr\ast}\colon u\mapsto f_u$, where $f_u(v)\bydef B(u,v)$ for all $v\in V$, sends $V$ onto $W$;
\item[(b)] for any $r\in\cR$, $\vphi(r)$ is the adjoint to $r$ with respect to $B$, i.e., $B(ru,v)=B(u,\vphi(r)v)$, for all $u,v\in V$.
\end{enumerate}
If $\vphi'_0$ is an antiautomorphism of $\cD$ and $B'$ is a $\vphi'_0$-sesquilinear form $V\times V\to\cD$ that define $W$ and $\vphi$ as in (a) and (b), then there exists a nonzero homogeneous $d\in\cD$ such that $B'=dB$ and $\vphi'_0(x)=d\vphi_0(x)d^{-1}$ for all $x\in\cD$.

As a partial converse, if $\vphi_0$ is an antiautomorphism of the graded algebra $\cD$ and $B:V\times V\to\cD$ is a weakly Hermitian nondegenerate homogeneous $\vphi_0$-sesquilinear form, then the adjoint with respect to $B$ defines an antiautomorphism $\vphi$ of the $G$-graded algebra $\CL{V}{W}$, with $W=\{f_u\;|\;u\in V\}$, such that $\vphi(\CF{V}{W})=\CF{V}{W}$.
\end{theorem}

\begin{proof}
Given an antiautomorphism $\vphi$, the existence of the pair $(\vphi_0,B)$ is already proved. If $(\vphi'_0,B')$ is another such pair, then the corresponding mapping $u\mapsto f'_u$ is an isomorphism of graded $\cR$-modules $V^{[g'_0]}\to W^\vphi$. Hence $\vphi'_1\circ\vphi_1^{-1}$ is a nonzero homogeneous element of $\Endgr_\cR(W)$, so there exists a nonzero homogeneous $d\in\cD$ such that $\vphi'_1(v)=d\vphi_1(v)$ for all $v\in V$, which implies $B'=dB$. Now the equation $\vphi'_0(x)=d\vphi_0(x)d^{-1}$ follows easily from \eqref{eq:sesquilinear}.

Conversely, for a given antiautomorphism $\vphi_0$ and a form $B$ of degree $g_0$, define $\vphi_1\colon V\to W$ by setting $\vphi_1(u)=f_u$. This is a homogeneous $\vphi_0$-semilinear isomorphism of degree $g_0$. Take $\vphi_2=Q^{-1}\circ\vphi_1^{-1}$. Then one checks that $(\vphi_0,\vphi_1,\vphi_2)$ is an isomorphism from $(\cD,V^{[g_0]},W^{[g_0^{-1}]})$ to $(\cD^{\mathrm{op}},W,V)$, so Theorem \ref{isomorphism_graded_simple} tells us that $\vphi(r)\bydef \vphi_1\circ r\circ\vphi_1^{-1}$ defines an isomorphism of graded algebras $\CL{V}{W}\to\CL{W}{V}=\CL{V}{W}^{\mathrm{op}}$ that restricts to an isomorphism $\CF{V}{W}\to\CF{W}{V}=\CF{V}{W}^{\mathrm{op}}$. It remains to observe that the definition of $\vphi_1$ implies that $\vphi(r)$ is the adjoint to $r$ with respect to $B$, for any $r\in\CL{V}{W}$. 
\end{proof}

Note that it follows that any antiautomorphism $\vphi$ of the graded algebra $\cR$ extends to an antiautomorphism of  $\CL{V}{W}$ and restricts to an antiautomorphism of $\CF{V}{W}$.

\begin{remark}
It is easy to compute $\vphi$ on $\CF{V}{W}$ explicitly. Indeed, equation \eqref{eq:psi_on_CF} now takes the form $\vphi(v\ot f_u)=\vphi_2(f_u)\ot\vphi_1(v)$, so, taking into account $\vphi_1(v)=f_v$ and $\vphi_2=Q^{-1}\circ\vphi_1^{-1}$, we obtain:
\begin{equation}\label{eq:vphi_on_CF}
\vphi(v\ot f_u)=Q^{-1}u\ot f_v\quad\mbox{for all}\quad u,v\in V.
\end{equation} 
\end{remark}

\subsection{Antiautomorphisms that are involutive on the identity component}\label{ssAAIIC}

We restrict ourselves to the case where $\cR$ is a locally finite simple algebra with minimal left ideals over an {\em algebraically closed} field $\FF$. If $\cR$ is given a grading by an abelian group $G$, then, by Theorem \ref{lfd}, $\cR$ is isomorphic to some $\frF(G,\cD,\tV,\tW)$ where $\cD$ is a matrix algebra with a division grading. Suppose that the graded algebra $\cR$ admits an antiautomorphism $\vphi$. Then, by Theorem \ref{antiauto_graded_simple}, we obtain an antiautomorphism $\vphi_0$ for $\cD$. It is known \cite{BZ06} that this forces the support $T$ of $\cD$ to be an elementary $2$-group and hence $\chr{\FF}\ne 2$ or $T=\{e\}$. From now on, we assume $\chr{\FF}\ne 2$. 

Since $G$ is abelian, any division grading on a matrix algebra can be realized using generalized Pauli matrices. If the support $T$ is an elementary $2$-group, then the matrix transpose preserves this grading. Choose a nonzero element $X_t$ in each component $\cD_t$. Then the transpose of $X_t$ equals $\beta(t)X_t$ where $\beta(t)\in\{\pm 1\}$. It is easy to check (see \cite{BK}) that $\beta\colon T\to\{\pm 1\}$ is a quadratic form on $T$ if we regard it as a vector space over the field of two elements, with the nondegenerate alternating bicharacter $\beta\colon T\times T\to\FF^\times$ being the associated bilinear form: $\beta(tt')=\beta(t)\beta(t')\beta(t,t')$ for all $t,t'\in T$. It is easy to see that any automorphism of the graded algebra $\cD$ is a conjugation by some $X_t$. Hence $\vphi_0$ is given by $\vphi_0(X_{t'})=\beta(t')X_{t}^{-1}X_{t'}X_t$ for some $t\in T$. In particular, $\vphi_0$ is an {\em involution}. Hence, the isomorphism $Q$ associated to the $\vphi_0$-sesquilinear form $B$ in Theorem \ref{antiauto_graded_simple} is linear over $\cD$ and thus $Q$ is an invertible element of the identity component of $\lvw$. Adjusting $B$, we may assume without loss of generality that 
\[
\vphi_0(X_t)=\beta(t)X_t\quad\mbox{for all}\quad t\in T.
\]
This convention makes the choice of $B$ unique up to a scalar in $\FF^\times$. 

Assume that $\vphi$ restricts to an involution on $\cR_e$. Then $B$ has certain symmetry properties, which we are going to investigate now. In particular, $B$ is {\em balanced}, i.e., for any pair of homogeneous $u,v\in V$, we have
\[
B(u,v)=0\;\Leftrightarrow\;B(v,u)=0.
\]

Recall that $\frF(G,\cD,\tV,\tW)=\CF{V}{W}$ where $V$ and $W$ are constructed from $\tV$ and $\tW$ as follows. Select a transversal $S$ for $T$ and, for each $A\in G/T$, set $V_A=\tV_A\ot\cD$ and $W_A=\cD\ot\tW_A$, with the degree of the elements of $\tV_A\ot 1$ and $1\ot\tW_A$ set to be the unique element of $A\cap S$, which we denote by $\gamma(A)$. It will be convenient to identify $\tV_A$ with $\tV_A\ot 1$ and $\tW_A$ with $1\ot\tW_A$.
 
Using the definition of $\bar{B}$ and equation \eqref{eq:operator_Q}, we compute, for all $u,v\in V$ and $r\in\CL{V}{W}$:
\begin{align}\label{eq:vphi_square_}
B(u,\vphi^2(r)v)&=B(\vphi(r)u,v)=\vphi_0(\bar{B}(v,\vphi(r)u))=\vphi_0(B(Qv,\vphi(r)u))\nonumber \\
&=\vphi_0(B(rQv,u))=\bar{B}(u,rQv)=B(Qu,rQv).
\end{align}
Substituting $r=1$, we obtain $B(u,v)=B(Qu,Qv)$ for all $u,v\in V$ and hence $B(Qu,v)=B(u,Q^{-1}v)$. So we can continue \eqref{eq:vphi_square_} to obtain, for all $u,v\in V$, $B(u,\vphi^2(r)v)=B(u,Q^{-1}rQv)$. Therefore,
\begin{equation}\label{eq:vphi_square}
\vphi^2(r)=Q^{-1}rQ\quad\mbox{for all}\quad r\in\cR.
\end{equation}

Observe that the identity component $\cR_e$ is the direct sum of subalgebras $\cR^A$, $A\in G/T$, where $\cR^A$ consists of all operators in $\cR_e$ that map the isotypic component $V_A$ into itself and other isotypic components to zero. Clearly, $\cR^A$ is spanned by the operators of the form $w\ot v$ where $v\in \tV_A$ and $w\in\tW_{A^{-1}}$. Being homogeneous of degree $e$, $Q$ maps $V_A$ onto $V_A$. The restriction of $Q$ to $V_A$ is linear over $\cD$ and, by \eqref{eq:vphi_square}, commutes with all elements of $\cR^A$. It follows that  $Q$ acts on $V_A$ as a scalar $\lambda_A\in\FF^\times$. Now \eqref{eq:operator_Q} implies that $B$ is balanced, as claimed.

The fact that $B$ is balanced allows us to define the concept of orthogonality for homogeneous elements and for graded subspaces of $V$.
Since $B$ is homogeneous of degree, say, $g_0$, we have, for all $u\in V_{g_1}$ and $v\in V_{g_2}$, that $B(u,v)=0$ unless $g_0 g_1 g_2\in T$. It follows that $V_A$ is orthogonal to all isotypic components except $V_{g_0^{-1} A^{-1}}$, and hence the restriction of $B$ to $V_A\times V_{g_0^{-1} A^{-1}}$ is nondegenerate. It will be important to distinguish whether or not $A$ equals $g_0^{-1}A^{-1}$. 

If $g_0 A^2=T$, then the element $g_0\gamma(A)^2\in T$ does not depend on the choice of the transversal and will be denoted by $\tau(A)$. The restriction of $B$ to $V_A\times V_A$ is a nondegenerate $\vphi_0$-sesquilinear form over $\cD$. It is uniquely determined by its restriction to $\tV_A\times\tV_A$, which is a bilinear form over $\FF$ with values in $\cD_{\tau(A)}$. Set
\begin{equation}\label{eq:B_and_tB_1}
B(u,v)=\tB_A(u,v)X_{\tau(A)}\quad\mbox{for all}\quad u,v\in\tV_A\;\mbox{ where }\;g_0 A^2=T.
\end{equation}
Then $\tB_A$ is a nondegenerate bilinear form on $\tV_A$ with values in $\FF$. Setting $t=\tau(A)$ for brevity, we compute:
\begin{align*}
\tB_A(v,u)X_t&=B(v,u)=\vphi_0(B(Qu,v))=\vphi_0(B(\lambda_A u,v))\\
&=\vphi_0(\lambda_A\tB_A(u,v)X_t)=\lambda_A \tB_A(u,v)\vphi_0(X_t)=\lambda_A\beta(t)\tB_A(u,v)X_t,
\end{align*}
so $\tB_A(v,u)=\lambda_A\beta(t)\tB_A(u,v)$. 

If $g_0 A^2\ne T$, then we may assume without loss of generality that the transversal is chosen so that 
\begin{equation}\label{eq:choice_of_transveral}
g_0\gamma(A)\gamma(g_0^{-1} A^{-1})=e. 
\end{equation}
Then the restrictions of $B$ to $\tV_A\times \tV_{g_0^{-1}A^{-1}}$ and to $\tV_{g_0^{-1}A^{-1}}\times \tV_A$ are nondegenerate bilinear forms with values in $\cD_e=\FF$. Denote them by $\tB_A$ and $\tB_{g_0^{-1}A^{-1}}$, respectively, i.e., set 
\begin{equation}\label{eq:B_and_tB_2}
B(u,v)=\tB_A(u,v)1\quad\mbox{for all}\quad u\in\tV_A\;\mbox{ and }\;v\in\tV_{g_0^{-1}A^{-1}}\;\mbox{ where }\;g_0 A^2\ne T.
\end{equation}
It is easy to see how $\tB_A$ and $\tB_{g_0^{-1}A^{-1}}$ are related: $\tB_{g_0^{-1}A^{-1}}(v,u)=\lambda_A\tB_A(u,v)$.

Putting all pieces together, we set $\tV_{\{A,g_0^{-1}A^{-1}\}}\bydef\tV_A\oplus\tV_{g_0^{-1}A^{-1}}$ and  
\begin{equation}\label{eq:tV_for_vphi}
\tV\bydef\bigoplus_{A\in G/T,\, g_0 A^2=T}\tV_A\quad\oplus\quad
\bigoplus_{\{A,g_0^{-1}A^{-1}\}\subset G/T,\, g_0 A^2\ne T}\tV_{\{A,g_0^{-1}A^{-1}\}},
\end{equation}
and define a nondegenerate bilinear form $\tB\colon\tV\times\tV\to\FF$ so that all summands in \eqref{eq:tV_for_vphi} are orthogonal to each other, the restriction of $\tB$ to $\tV_A\times\tV_A$ is $\tB_A$ if $g_0 A^2=T$ while the restriction of $\tB$ to $\tV_A$ is zero and the restriction to $\tV_A\times \tV_{g_0^{-1}A^{-1}}$ is $\tB_A$ if $g_0 A^2\ne T$. 

Conversely, let $\tV$ be a vector space over $\FF$ that is given a grading by $G/T$ and let $\tB$ be a nondegenerate bilinear form on $\tV$ that is compatible with the grading in the sense that
\[
\tB(\tV_{g_1 T},\tV_{g_2 T})=0\quad\mbox{for all}\quad g_1,g_2\in G\;\mbox{ with }\;g_0 g_1 g_2 T\ne T
\]
and, for all $A\in G/T$, satisfies the following symmetry condition:
\begin{equation}\label{eq:def_mu}
\tB(v,u)=\mu_A\tB(u,v)\quad\mbox{for all}\quad u\in\tV_A\;\mbox{ and }\;v\in\tV_{g_0^{-1}A^{-1}}\quad\mbox{where}\quad\mu_A\in\FF^\times.
\end{equation}
It follows that $\mu_A \mu_{g_0^{-1}A^{-1}}=1$. Hence, if $g_0 A^2=T$, then $\tV$ restricts to a symmetric or a skew-symmetric form on $\tV_A$. 
%We will write $\delta(A)=1$ in the first case and $\delta(A)=-1$ in the second case (i.e., $\delta(A)=\mu_A$). 
For any $A\in G/T$, let $\tB_A$ be the restriction of $\tB$ to $\tV_A\times\tV_{g_0^{-1}A^{-1}}$. It follows that $\tB_A$ is nondegenerate. Choose a transversal $S$ for $T$ so that \eqref{eq:choice_of_transveral} holds for all $A$ with $g_0 A^2\ne T$. Set $V_A=\tV_A\ot{}^{[\gamma(A)]}\cD$. Then $V_A$ is a graded right $\cD$-module whose isomorphism class does not depend on the choice of $S$. Set $V=\bigoplus_{A\in G/T}V_A$. Define $B\colon V\times V\to\cD$ using \eqref{eq:B_and_tB_1} and \eqref{eq:B_and_tB_2}, setting $B$ equal to zero in all other cases and then extending by $\vphi_0$-sesquilinearity. Clearly, $B$ is nondegenerate. Set $W=\{f_u\;|\;u\in V\}$ where $f_u(v)=B(u,v)$. We will denote the corresponding $G$-graded algebra $\fvw$ by $\frF(G,\cD,\tV,\tB,g_0)$ or $\frF(G,T,\beta,\tV,\tB,g_0)$, since $\cD$ is determined by the support $T$ and bicharacter $\beta$. The graded algebra $\fvw$ has an antiautomorphism $\vphi$ defined by the adjoint with respect to $B$. Indeed, let $Q\colon V\to V$ act on $V_A$ as the scalar $\lambda_A$ where
\begin{equation}\label{eq:lambda_and_mu}
\lambda_A=\left\{
\begin{array}{ll}
\mu_A\beta(\tau(A)) & \mbox{if}\quad g_0 A^2=T;\\
\mu_A & \mbox{if}\quad g_0 A^2\ne T.
\end{array}\right.
\end{equation}
Then $Q$ satisfies $\bar{B}(u,v)=B(Qu,v)$ for all $u,v\in V$ and hence $B$ is weakly Hermitian. Since $Q$ commutes with the elements of $\cR_e$, equation \eqref{eq:vphi_square} tells us that $\vphi^2$ is the identity on $\cR_e$.

\begin{df}\label{df:datum_finitary_2}
With fixed $\cD$ and $\vphi_0$, we will write $(\tV,\tB,g_0)\sim (\tV',\tB',g'_0)$ if there is an element $g\in G$ such that $g'_0=g_0 g^{-2}$ and, for any $A\in G/T$ with $g_0 A^2=T$, we have $\tV_{A}\cong\tV'_{gA}$ as inner product spaces while, for any $A\in G/T$ with $g_0 A^2\ne T$, we have $(\tV_{A},\tV_{g_0^{-1}A^{-1}})\cong(\tV'_{gA},\tV'_{(g'_0)^{-1}g^{-1}A^{-1}})$ and $\mu_A=\mu'_{gA}$ (where $\mu_A$ is defined by \eqref{eq:def_mu} and $\mu'_{A}$ by the same equation with $\tB$ replaced by $\tB'$). 
\end{df}

Recall that, disregarding the grading, $\cR$ can be represented as $\frF_\Pi(U)$ with $U=\tV\ot N$ and $\Pi=M\ot\tW$, where $M$ and $N$ are the natural right and left modules for $\cD=M_\ell(\FF)$, respectively (see the analysis preceding Theorem \ref{lfd}). In our case, $W=\{f_u\;|\;u\in V\}$ and $\tW=\{f_{\tilde u}\;|\;\tilde u\in \tV\}$. Note that $f_{\tilde u}(\tilde v)=B(\tilde u,\tilde v)$ does not necessarily belong to $\cD_e=\FF$ for all $\tilde u,\tilde v\in\tV$, so equation \eqref{eq:pairing_U_Pi} for the $\FF$-bilinear form $\Pi\times U\to\FF$ should be modified as follows:
\begin{equation}\label{eq:pairing_U_Pi_2}
(m\ot\tilde u,\tilde v\ot n)=mB(\tilde u,\tilde v)n\quad\mbox{for all}\quad m\in M, n\in N,\, \tilde u,\tilde v\in \tV.
\end{equation}
If we identify $U$ and $\Pi$ with $\tV^\ell$, then the above $\FF$-bilinear form is given by
\begin{equation}\label{eq:two_bilinear_forms_2}
((\tilde u_1,\ldots,\tilde u_\ell),(\tilde v_1,\ldots,\tilde v_\ell))=\sum_{i,j=1}^\ell x_{ij}\tB(\tilde u_i,\tilde v_j)\quad\mbox{for all}\quad \tilde u_i\in\tV_{A_i}\;\mbox{ and }\;\tilde v_j\in\tV_{A'_j},
\end{equation}
where $x_{ij}$ is the $(i,j)$-entry of the matrix $X_{\tau(A)}$ if $A_i=A'_j=A$ with $g_0 A^2=T$ and $x_{ij}=\delta_{i,j}$ otherwise.

Now we are ready to state the result:

\begin{theorem}\label{lfd_2}
Let $\cR$ be a locally finite simple algebra with minimal left ideals over an algebraically closed field $\FF$, $\chr{\FF}\ne 2$. If $\cR$ is given a grading by an abelian group $G$ and an antiautomorphism $\vphi$ that preserves the grading and restricts to an involution on $\cR_e$, then $(\cR,\vphi)$ is isomorphic to some $\frF(G,\cD,\tV,\tB,g_0)$ where $\cD$ is a matrix algebra over $\FF$ equipped with a division grading and an involution $\vphi_0$ given by matrix transpose. Conversely, if $\cD=M_\ell(\FF)$ with a division grading and  $\vphi_0$ is the matrix transpose, then $\frF(G,\cD,\tV,\tW)$ is a locally finite simple algebra with minimal left ideals, which can be represented as $\frF_U(U)$ where $U=\tV^\ell$ and the nondegenerate bilinear form $U\times U\to\FF$ is given by \eqref{eq:two_bilinear_forms_2}. Two such graded algebras $\frF(G,\cD,\tV,\tB,g_0)$ and $\frF(G,\cD',\tV',\tB',g'_0)$ are not isomorphic unless $\cD\cong\cD'$ as graded algebras, whereas, for fixed $\cD$ and $\vphi_0$, $\frF(G,\cD,\tV,\tB,g_0)$ and $\frF(G,\cD,\tV',\tB',g'_0)$ are isomorphic as graded algebras with antiautomorphism if and only if $(\tV,\tB,g_0)\sim (\tV',\tB',g'_0)$ in the sense of Definition \ref{df:datum_finitary_2}.
\end{theorem}

\begin{proof}
The first two assertions are already proved. Let $\cR=\frF(G,\cD,\tV,\tB,g_0)$ and $\cR'=\frF(G,\cD',\tV',\tB',g'_0)$. Denote their antiautomorphisms by $\vphi$ and $\vphi'$, respectively. By Theorem \ref{isomorphism_graded_simple}, if $\psi\colon\cR\to\cR'$ is an isomorphism of graded algebras, then, for some $g\in G$, we have an isomorphism of triples $(\psi_0,\psi_1,\psi_2)$ from $(\cD,V^{[g]},W^{[g^{-1}]})$ to $(\cD',V',W')$, so $\cD\cong\cD'$ as graded algebras. Now suppose $\cD=\cD'$ and  $\vphi'=\psi\circ\vphi\circ\psi^{-1}$. If $\vphi$ and $\vphi'$ are given by the adjoint with respect to $\vphi_0$-sesquilinear forms $B$ and $B'$, respectively, then the form $B''(u',v')\bydef\psi_0(B(\psi_1^{-1}(u'),\psi_1^{-1}(v')))$, $u',v'\in V'$, is also  $\vphi_0$-sesquilinear (because $\psi_0$ commutes with $\vphi_0$) and gives $\vphi'$ as adjoint (because $\psi(r)=\psi_1\circ r\circ\psi_1^{-1}$ for all $r\in\cR$). It follows that $B''=\lambda B$ for some $\lambda\in\FF^\times$. The degree of $B''$ is $g_0 g^{-2}$, so we obtain $g'_0=g_0 g^{-2}$. For each $A\in G/T$, $\psi_1$ restricts to an isomorphism $\psi_A\colon\tV_{A}\to\tV'_{gA}$. Consider the bilinear form $\tB''_{gA}\bydef\tB_A\circ(\psi_A^{-1}\times\psi_{g_0^{-1}A^{-1}}^{-1})$. For any $A$ with $g_0 A^2\ne T$, we have  $\tB''_{gA}=\lambda\tB'_{gA}$, so $(\psi_A,\psi_{g_0^{-1}A^{-1}})$ is an isomorphism of pairs and $\mu_A=\mu'_{gA}$. For any $A$ with $g_0 A^2=T$, $\tB''_{gA}$ is still a scalar multiple of $\tB'_{gA}$ (because $\psi_0(X_t)$ is a scalar multiple of $X_t$ for all $t\in T$), so $\psi_A$ preserves inner product up to a scalar multiple. Since $\FF$ is algebraically closed, we obtain $\tV_{A}\cong\tV'_{gA}$ as inner product spaces.

Conversely, if $(\tV,\tB,g_0)\sim (\tV',\tB',g'_0)$, then we can construct a $\cD$-linear isomorphism $\psi_1\colon V^{[g]}\to V'$ such that $B'=B\circ(\psi_1^{-1}\times\psi_1^{-1})$. Define $\psi_2(f_u)=f_{\psi_1(u)}$ for all $u\in V$. Then $(\id,\psi_1,\psi_2)$ is an isomorphism from $(\cD,V^{[g]},W^{[g^{-1}]})$ to $(\cD,V',W')$. By Theorem \ref{isomorphism_graded_simple}, we obtain an isomorphism of graded algebras  $\psi\colon\cR\to\cR'$. By construction, $\vphi'=\psi\circ\vphi\circ\psi^{-1}$.
\end{proof}

In the case where $\cR$ has finite or countable dimension, i.e., $\cR=M_n(\FF)$ with $n\in\NN\cup\{\infty\}$, we can express the classification in combinatorial terms. Since $\FF$ is algebraically closed, two vector spaces with symmetric inner products are isomorphic if they have the same finite or countable dimension. The same is true for vector spaces with skew-symmetric inner product. Therefore, for $A\in G/T$ with $g_0A^2=T$, the isomorphism class of $\tV_A$ is encoded by $\mu_A$ and $\dim\tV_A$. For $A\in G/T$ with $g_0A^2\ne T$, the isomorphism class of $(\tV_{A},\tV_{g_0^{-1}A^{-1}})$ is encoded by $\dim\tV_A=\dim\tV_{g_0^{-1}A^{-1}}$. We introduce functions $\mu\colon G/T\to\FF^\times$ sending $A$ to $\mu_A$ (where we set $\mu_A=1$ if $\tV_A=0$) and, as before, $\kappa\colon G/T\to\{0,1,2,\ldots,\infty\}$ sending $A$ to $\dim\tV_A$. Recall that $\mu$ satisfies $\mu_A\mu_{g_0^{-1}A^{-1}}=1$ for all $A\in G/T$ and $\kappa$ has a finite or countable support. We will denote the associated graded algebra with antiautomorphism $\frF(G,\cD,\tV,\tB,g_0)$ by $\frF(G,\cD,\kappa,\mu,g_0)$ or by $\frF(G,T,\beta,\kappa,\mu,g_0)$. For a given elementary $2$-subgroup $T\subset G$ and a bicharacter $\beta$, we fix a realization of $\cD$ using Pauli matrices and thus fix an involution $\vphi_0$ on $\cD$.  Finally, for any $g\in G$, define $\mu^g$ and $\kappa^g$ by setting $\mu^g(gA)\bydef\mu(A)$ and $\kappa^g(gA)\bydef\kappa(A)$ for all $A\in G/T$.

\begin{corollary}\label{gradings_on_M_infinity_2}
Let $\FF$ be an algebraically closed field, $\chr{\FF}\ne 2$, and let $\cR=M_n(\FF)$ where $n\in\NN\cup\{\infty\}$. If $\cR$ is given a grading by an abelian group $G$ and an antiautomorphism $\vphi$ that preserves the grading and restricts to an involution on $\cR_e$, then $(\cR,\vphi)$ is isomorphic to some $\frF(G,\cD,\kappa,\mu,g_0)$ where $\cD=M_\ell(\FF)$, with $\ell\in\NN$ and $n=|\kappa|\ell$, is equipped with a division grading and an involution $\vphi_0$ given by matrix transpose. Moreover, $\frF(G,T,\beta,\kappa,\mu,g_0)$ and $\frF(G,T',\beta',\kappa',\mu',g'_0)$ are isomorphic as graded algebras with antiautomorphism if and only if $T'=T$, $\beta'=\beta$, and  there exists $g\in G$ such that $g'_0=g_0 g^{-2}$, $\kappa'=\kappa^g$ and $\mu'=\mu^g$.\qed
\end{corollary}

\subsection{Involutions}\label{ssI}

We can specialize Theorem \ref{lfd_2} to obtain a classification of involutions on the graded algebra $\cR$. By \eqref{eq:vphi_square}, $\vphi$ is an involution if and only if $Q\colon V\to V$ is a scalar operator, i.e., $\bar{B}=\lambda B$ for some $\lambda\in\FF^\times$, which implies $\lambda\in\{\pm 1\}$. Disregarding the grading, $\cR=\frF_U(U)$ where $U$ is an inner product space. Since $\vphi_0$ is matrix transpose, equation \eqref{eq:pairing_U_Pi_2} implies that $(v,u)=\lambda(u,v)$ for all $u,v\in U$. Hence, in the case $\bar{B}=B$, we obtain an {\em orthogonal involution} on $\cR$ and write $\sgn(\vphi)=1$, whereas in the case $\bar{B}=-B$, we obtain a {\em symplectic involution} and write $\sgn(\vphi)=-1$. Since all $\lambda_A$ must be equal to $\lambda=\sgn(\vphi)$, equation \eqref{eq:lambda_and_mu} yields
\begin{equation}\label{eq:mu_for_inv}
\mu_A=\left\{
\begin{array}{ll}
\sgn(\vphi)\beta(\tau(A)) & \mbox{if}\quad g_0 A^2=T;\\
\sgn(\vphi) & \mbox{if}\quad g_0 A^2\ne T.
\end{array}\right.
\end{equation}
We note that, for $g_0 A^2\ne T$, although the space $\tV_{\{A,g_0^{-1}A^{-1}\}}$ now has a symmetric or skew-symmetric inner product, the equivalence relation in Definition \ref{df:datum_finitary_2} requires more than just an isomorphism of inner product spaces: the isomorphism must respect the direct sum decomposition $\tV_{\{A,g_0^{-1}A^{-1}\}}=\tV_A\oplus\tV_{g_0^{-1}A^{-1}}$. To summarize:

\begin{proposition}\label{lfd_inv}
Under the conditions of Theorem \ref{lfd_2}, $\vphi$ is an involution if and only if $\tB$ satisfies the symmetry condition \eqref{eq:def_mu} where $\mu_A$ is given by \eqref{eq:mu_for_inv}.\qed
\end{proposition}

In the case of finite or countable dimension, we again can reduce everything to combinatorial terms (for the finite case, this result appeared in \cite{BK}). Once $\delta=\sgn(\vphi)$ is specified, the function $\mu\colon G/T\to\FF^\times$ is determined by \eqref{eq:mu_for_inv}, so we will denote the corresponding graded algebra with involution by $\frF(G,\cD,\kappa,\delta,g_0)$ or $\frF(G,T,\beta,\kappa,\delta,g_0)$.

\begin{corollary}\label{gradings_on_M_infinity_inv}
Let $\cR=\frF(G,T,\beta,\kappa)$ where the ground field $\FF$ is algebraically closed, $\chr{\FF}\ne 2$, and $G$ is an abelian group. The graded algebra $\cR$ admits an involution $\vphi$ with $\sgn(\vphi)=\delta$ if and only if $T$ is an elementary $2$-group and, for some $g_0\in G$, we have $\kappa(A)=\kappa(g_0^{-1}A^{-1})$ for all $A\in G/T$ and we also have $\beta(g_0 a^2)=\delta$ for all $A=aT\in G/T$ such that $g_0A^2=T$ and $\kappa(A)$ is finite and odd. If $\vphi$ is an involution on $\cR$ with $\sgn(\vphi)=\delta$, then the pair $(\cR,\vphi)$ is isomorphic to some $\frF(G,T,\beta,\kappa,\delta,g_0)$. Moreover, $\frF(G,T,\beta,\kappa,\delta,g_0)$ and $\frF(G,T',\beta',\kappa',\delta',g'_0)$ are isomorphic as graded algebras with involution if and only if $T'=T$, $\beta'=\beta$, $\delta'=\delta$, and there exists $g\in G$ such that $g'_0=g_0 g^{-2}$ and $\kappa'=\kappa^g$.\qed
\end{corollary}

%-------------------------------------------------------------------------------

\section{Functional identities and Lie homomorphisms}\label{sFILH}

The so-called {\em Herstein's Lie map conjectures} describe the relationship between (anti)isomorphisms (respectively, derivations) of certain associative algebras and isomorphisms (respectively, derivations) of related Lie algebras. These conjectures were proved and vastly generalized using the technique of {\em functional identities} --- see \cite{FIbook} and references therein. For example, consider  $\cR_i=\frF_{\Pi_i}(U_i)$ and $\cL_i=[\cR_i,\cR_i]=\fsl(U_i,\Pi_i)$, $i=1,2$, where $U_i$ is an infinite-dimensional vector space over a field $\FF$ and $\Pi_i\subset U_i^\ast$ is a total subspace. Clearly, if $\psi\colon\cR_1\to\cR_2$ is an isomorphism (respectively, antiisomorphism) of associative algebras, then $\psi$ (respectively, $-\psi$) restricts to an isomorphism of Lie algebras $\cL_1\to\cL_2$. The converse is a special case of Herstein's conjecture. It follows that the Lie algebras $\fsl(U_1,\Pi_1)$ and $\fsl(U_2,\Pi_2)$ are isomorphic if and only if $(U_1,\Pi_1)$ is isomorphic to $(U_2,\Pi_2)$ or to $(\Pi_2,U_2)$. Similarly, $\fso(U_1,\Phi_1)$ and $\fso(U_2,\Phi_2)$ (respectively, $\fsp(U_1,\Phi_1)$ and $\fsp(U_2,\Phi_2)$) are isomorphic if and only if $(U_1,\Phi_1)$ is isomorphic to $(U_2,\Phi_2)$. In order to prove that all gradings on a Lie algebra $\cL=\fsl(U,\Pi)$, $\fso(U,\Phi)$ or $\fsp(U,\Phi)$ by an abelian group $G$ come, in some sense, from gradings on the associative algebra $\cR=\fpu$ (with $\Pi=U$ in the last two cases), we will have to describe  surjective Lie homomorphisms $\cL\ot\cH\to\cL\ot\cH$ where $\cH=\FF G$, the group algebra of $G$ over $\FF$. This was done in \cite{BB} under the restriction that $G$ is finite. Our goal now is to remove this restriction.

We first introduce the notation for this section. Let $U$ be a vector space over $\FF$ and let $\Pi\subset U^\ast$ be a total subspace, i.e., we have a nondegenerate bilinear form $(\,,\,):\Pi\times U\to \FF$. Let $\cR=\fpu$. Clearly, $\cR$ is a subalgebra of the algebra $\End(U)$ of all linear operators on $U$. Further, let $\cH$ be an arbitrary unital algebra over $\FF$. We consider  $U\otimes \cH$ as a right $\cH$-module via $(u\otimes k)h = u\otimes kh$, and set $\cQ=\End_\cH(U\otimes \cH)$. We can imbed $\cS\bydef\cR\otimes \cH$ into $\cQ$ via $r\otimes h \mapsto T_{r,h}$, $r\in \cR$, $h\in \cH$, where
\[
T_{r,h}(u\otimes k) = ru\otimes hk.
\]
By abuse of notation, we identify $\cS$ with its image in $\cQ$. Similarly, we identify the elements $r$ in $\cR$ with the operators $T_{r,1}$, so we have $\cR\subset\cS\subset\cQ$. The center of  $\cQ$ will be denoted by $\cC$. It is easy to check that $\cC$ consists of operators of the form $u\otimes k\mapsto u\otimes ck$, where $c$ is in the center of $\cH$, and is equal to the centralizer of $\cS$ in $\cQ$. Thus, we may identify $\cC$ with the center of $\cH$. 

Given an algebra $\cA$, we denote by $M(\cA)$  its {\em  multiplication algebra}. This is the algebra of all maps on $\cA$ of the form $x\mapsto \sum_i a_i x b_i$ with $a_i,b_i\in \cA$. Of course, such maps are defined also on every algebra $\cA'$ that contains $\cA$ as a subalgebra. In this way we may consider $M(\cA)$ as a subalgebra of $M(\cA')$.

\subsection{Preliminary results}\label{ssPR} 

We will now establish several auxiliary results needed for the treatment of functional identities on the algebra $\cS$. 

We begin by a lemma concerning the multiplication algebra of $\cR$.
Let $a_1,\ldots,a_n\in \cR$. Then there exists an idempotent $\veps\in\cR$ such that $a_i\in\veps\cR\veps$ for all $i$ and $\veps\cR\veps  \cong M_n(\FF)$ for some $n\ge 1$. But then, as it is well known and easy to see,  $M(\veps\cR\veps) = \End(\veps\cR\veps)$. Therefore, if $a_1,\ldots,a_n$ are linearly independent, there exists $f\in M(\veps\cR\veps)$ such that 
$f(a_1) \ne 0$ and $f(a_2) = \ldots = f(a_n) =0$. We actually need a slightly different result, namely:

\begin{lemma}\label{Lel}
Let $a_0,a_1,\ldots,a_n\in \cR\cup \{1\}(\subseteq \cQ)$ be linearly independent. 
Then, for each $i$, there exists $f\in M(\cR)$ such that $0\ne f(a_i)\in\cR$ and $f(a_j) =0$ for all $j\ne i$. 
\end{lemma} 

\begin{proof}
If all $a_i\in \cR$, then this follows from the above argument. We may therefore assume that $a_0=1\notin \cR$. Let $\veps\in \cR$ be an idempotent such that $a_i=\veps a_i\veps$, $1\le i\le n$. We may pick a rank one idempotent $\veps_1\in \cR$ such that $\veps\veps_1=\veps_1\veps=0$. Indeed, since $1\notin \cR$, $U$ must be infinite-dimensional, and the existence of such $\veps_1$ can be easily shown (cf. \cite[Theorem 4.3.1]{BMMb}). Let $\veps'\in \cR$ be an idempotent such that all $\veps,\veps_1,a_1,\ldots,a_n$ are contained in $\veps'\cR \veps'$ and $\veps'\cR\veps'\cong M_m(\FF)$ for some $m\ge 1$. Note that, for every  $f\in M(\veps'\cR\veps')$, we have $f(\veps') = f(1)$. Since $\veps',a_1,\ldots,a_n$ are linearly independent, the existence of $f\in M(\veps'\cR\veps')$ satisfying the desired conditions can now be proved as above.
\end{proof}

\begin{lemma}\label{5SL2}
Let $a_0,a_1,\ldots,a_n\in \cR\cup \{1\}$ be linearly independent. If $q_0,q_1,\ldots,q_n \in \cQ$ are such that  
$\sum_{j=0}^nq_jxa_j = 0$ for all $x\in \cR$, then $q_j=0$ for all $j$. 
Similarly, $\sum_{j=0}^na_jxq_j = 0$ for all $x\in \cR$ implies $q_j=0$ for all $j$.
\end{lemma}
 
\begin{proof}
Pick $i\in\{0,1,\ldots,n\}$. By Lemma \ref{Lel}, there exists $f\in M(\cR)$, $f\colon x\mapsto \sum b_kxc_k$ for some $b_k,c_k\in \cR$, such that $a=f(a_i)$ is a nonzero element in $\cR$ and $f(a_j) =0$ whenever $j\ne i$. Let us replace $x$ by $xb_k$ in $\sum_{i=0}^nq_ixa_{i} = 0$, and then right-multiply the identity so obtained by $c_k$. Summing up over all $k$, we obtain $0 =\sum_j q_j x f(a_j)= q_ixa$ for every $x\in \cR$. The simplicity of $\cR$ implies $q_i\cR =0$, which clearly yields $q_i=0$.

The condition $\sum_{j=0}^na_jxq_j = 0$ can be treated analogously.
\end{proof}

\begin{lemma}\label{5SL1}
Let $a,b\in \cR$ and $p,q\in \cQ$ be such that $axq=pxb$ for all $x\in \cS$. If $a\ne 0$ and $b\ne 0$, then there
exists $\lambda\in \cC$ such that $p=\lambda a$ and $q=\lambda b$. 
\end{lemma}

\begin{proof}
Pick $z\in U$ such that $b(z\otimes 1)\ne 0$. Since $b\in \cR$, we have $b(z\otimes 1) = y\otimes 1$ for some nonzero $y\in U$. Now choose $w\in\Pi$ so that $(w,y) =1$, and let $x = T_{t,h}\in \cS$ where $h$ is an arbitrary element of $\cH$ and $t\in \cR$ is given by $tu=(w,u)v$, with $v$ an arbitrary element of $U$. The identity $(axq)(z\otimes 1) = (pxb)(z\otimes 1)$ now implies that there is $c\in \cH$ such that $p(v\otimes h) = a(v\otimes hc)$ for all $v\in U$ and $h\in \cH$. As $p$ and $a$ are $\cH$-linear, this yields 
\[
a(v\otimes 1)hc = a(v\ot hc) = p(v\otimes 1)h = a(v\otimes c)h = a(v \otimes 1)ch.
\]
Since $a$ is a nonzero element of $\cR$, this implies $hc = ch$ for all $h\in \cH$. Thus we have indeed $p=\lambda a$ where $\lambda\colon u\otimes k\mapsto u\otimes ck$ belongs to $\cC$. The identity $axq=pxb$ can now be rewritten as $ax(q-\lambda b) =0$, and therefore $q=\lambda b$ by Lemma \ref{5SL2}. 
\end{proof}

\begin{lemma}\label{Lglavna}
Let $a_0,a_1,\ldots,a_n\in \cR\cup \{1\}$ be linearly independent. If $p_i,q_i\in \cQ$ are such that 
$\sum_{i=0}^{n}(p_ixa_{i} + a_i x q_i) = 0$ for all $x\in \cS$, then all $p_i$ and $q_i$ lie in $\sum_{i=0}^n \cC a_i$.
\end{lemma}

\begin{proof} 
Let $f\in M(\cR)$. Arguing similarly to the proof of Lemma \ref{5SL2}, we see that the given identity implies 
$\sum_{i=0}^{n}(f(p_i)xa_{i} + f(a_i) x q_i) = 0$ for all $x\in \cS$. 
By Lemma \ref{Lel}, we can choose $f$ so that the last identity can be written as
\begin{equation}\label{ee1}
\sum_{i=0}^{n}p_i'xa_{i} + axq_0  = 0 \quad\mbox{for all}\quad x\in \cS,
\end{equation}
where $0\ne a\in \cR$ and $p_i'\in \cQ$. For any $g\in M(\cR)$, this yields $\sum_{i=0}^{n}p_i'xg(a_i) + axg(q_0) = 0$ for all $x\in \cS$. Invoking Lemma \ref{Lel} again, we see that, for any fixed $i$, we can choose $g$ so that the last identity reduces to $axq = p_i'x b$ with $q\in \cQ$ and $0\ne b\in \cR$. Now Lemma \ref{5SL1} tells us that $p_i' = \lambda_i a$ for some $\lambda_i\in \cC$. Accordingly, \eqref{ee1} can be written as $ax(q_0 + \sum_{i=0}^n \lambda_i a_i) = 0$ for all $x\in \cS$. This yields $q_0 = - \sum_{i=0}^n \lambda_i a_i \in \sum_{i=0}^n \cC a_i$ by Lemma \ref{5SL2}. In a similar way, we obtain the same conclusion for the other $q_i$ and for all $p_i$.
\end{proof}

\begin{lemma} \label{Lbasic}
Let $c$ be a nonzero element of $\cR$. If $h\colon\cS\to \cQ$ is an linear map satisfying $h(xcy) = h(x)yc$ for all $x,y\in \cS$, then there exists $q\in \cQ$ such that $h(x)= qxc$ for all $x\in \cS$. 
\end{lemma}

\begin{proof}
Since $\cR$ is simple, the ideal of $\cR$ generated by $c$ is equal to $\cR$.  This readily implies that the ideal of $\cS$ generated by $c$ is equal to $\cS$.  Let us define $\varphi\colon\cS\to \cQ$ by $\varphi(\sum_i x_icy_i) = \sum_i h(x_i)y_i$, $x_i,y_i\in \cS$. To see that $\varphi$ is well-defined, assume that $\sum_i x_icy_i =0$ for some $x_i,y_i\in \cS$. Then $0 = h(\sum_i x_icy_i) = \sum_i h(x_icy_i) = \sum_i h(x_i)y_ic$. Lemma \ref{5SL2} tells us that then $\sum_i h(x_i)y_i = 0$ as desired. Obviously, $\varphi$ is a linear map satisfying $\varphi(xy) = \varphi(x)y$ for all $x,y\in \cS$. If there is $q\in \cQ$ such that $\varphi(x) = qx$ for every $x\in \cS$, then it will follow that $qxcy= \varphi(xcy) = h(x)y$ for all $x,y\in \cS$, yielding the the conclusion of the lemma. Let us therefore prove that such $q$ indeed exists.

Choose a basis $\{v_i\,|\,i\in I\}$ of $U$. For each $i$ pick $w_i\in\Pi$ such that $(w_i,v_i) =1$, and denote by $\veps_i$ the operator in $\cR$ given by $\veps_i(u)=(w_i,u) v_i$. We now define $q\in \cQ$ as the $\cH$-linear map determined by 
\[
q(v_i\otimes 1) = \varphi(\veps_i)(v_i\otimes 1)\quad\mbox{for all}\quad i\in I. 
\]
Now choose arbitrary $v\in U$, $w\in \Pi$, $h\in \cH$ and $i\in I$. Let $t_i\in \cR$ be defined by
$t_i(u) = (w,u) v_i$. We have
\begin{align*}
\varphi(T_{t_i,h})(v\otimes 1) &= \varphi(\veps_i T_{t_i,h})(v\otimes 1) =  \varphi(\veps_i)T_{t_i,h}(v\otimes 1) \\
&= (w,v) \varphi(\veps_i)(v_i\otimes h) = (w,v) q(v_i\otimes h) = (q T_{t_i,h})(v\otimes 1).
\end{align*}
Thus, $\varphi(T_{t_i,h}) = q T_{t_i,h}$. Since the elements of the form $T_{t_i,h}$ linearly span $\cS$, this shows that $\varphi(x) = qx$ for all $x\in \cS$.
\end{proof}

\subsection{Functional identities} 

Now we briefly cover the necessary background on functional identities. We refer the reader to the book \cite{FIbook} for details.

Let $\Q$ be a unital ring with center $\mathfrak{C}$ and let $\R$ be a nonempty subset of $\Q$. (We are actually interested in the situation where $\Q$ is the algebra $\cQ$ and $\R$ is its subalgebra $\cS$, but the definitions that follow make sense in this abstract setting.) For any positive integer $k$,  $\R^k$ will denote the Cartesian product of $k$ copies of $\R$. Let us fix a positive integer $m$. Given $x_1,\ldots,x_m\in \R$ and $1\le i < j \le m$, we will write
\begin{align*}
\ov{x}_m &= (x_1,\ldots,x_m)\in\R^{m},\\
\ov{x}_{m}^{i}&= (x_1,\ldots,x_{i-1},x_{i+1},\ldots,x_m)\in\R^{m-1},\\
\ov{x}_{m}^{ij}= \ov{x}_{m}^{ji} &=
(x_1,\ldots,x_{i-1},x_{i+1},\ldots,
x_{j-1},x_{j+1}\ldots,x_m)\in\R^{m-2}.
\end{align*} 
For functions $F\colon\R^{m-1}\to \Q$ and $p\colon\R^{m-2}\to \Q$, we will write $F^i$ instead of $F(\ov{x}_{m}^{i})$, and $p^{ij}$ instead of $p(\ov{x}_{m}^{ij})$. Let $a\in \Q$ be such that $a\R\subset \R$. For a function $H\colon\R^m\to \Q$, we will write
\[
H(x_ka)\quad\mbox{for}\quad H(x_1,\ldots,x_{k-1},x_ka,x_{k+1},\ldots,x_m),
\]
and, for $F\colon\R^{m-1}\to\Q$, we will write
\[
F^i(x_ka)\quad\mbox{for}\quad
\left\{\begin{array}{ll}
F(x_1,\ldots,x_{i-1},x_{i+1},\ldots,x_{k-1},x_ka,x_{k+1},\ldots,x_m)&\mbox{if}\quad i < k,\\
F(x_1,\ldots,x_{k-1},x_ka,x_{k+1},\ldots,x_{i-1},x_{i+1},\ldots,x_m)&\mbox{if}\quad i > k.
\end{array}\right.
\]
We define $H(ax_k)$ and $F^i(ax_k)$ in a similar way. 

If $m=1$ then a map $F^i$ (where necessarily $i=1$) should be understood as a constant, i.\,e., it can be identified with an element of $\Q$. Similarly, if $m=2$ then a map $p^{ij}$ (where necessarily $i=1$ and $j=2$) is just an element of $\Q$.

Throughout, $I$ and $J$  will be subsets of $\{1,2,\ldots,m\}$. Let 
\[
E_i,F_j\colon\R^{m-1}\to\Q,\,\, i\in I,\,\, j\in J,
\] 
be arbitrary functions. The basic functional identity is 
\[
\sum_{i\in I}E_i(\ov{x}_m^i)x_i+ \sum_{j\in J}x_jF_j(\ov{x}_m^j) = 0\quad\mbox{for all}\quad \ov{x}_m\in \R^m.
\]
Using the abbreviations introduced above, we can write this as
\begin{equation}\label{fi1}
\sum_{i\in I}E_i^i x_i+ \sum_{j\in J}x_jF_j^j = 0\quad\mbox{for all}\quad \ov{x}_m\in \R^m.
\end{equation}
A slightly weaker functional identity
\begin{equation}\label{fi2}
\sum_{i\in I}E_i^i x_i+ \sum_{j\in J}x_jF_j^j \in \mathfrak{C}\quad\mbox{for all}\quad \ov{x}_m\in \R^m
\end{equation}
is also important. The cases where $I=\varnothing$ or $J=\varnothing$ are not excluded; here we follow the convention that the sum over $\varnothing$ is $0$.

When studying a functional identity, the goal is to describe the functions involved. A natural situation when \eqref{fi1} --- and hence also \eqref{fi2} --- holds is the following: there exist maps
\begin{align*}
p_{ij}\colon & \R^{m-2}\to \Q,\quad i\in I,\, j\in J,\, i\not=j,\\
\lambda_k\colon & \R^{m-1}\to \mathfrak{C},\quad k\in I\cup J,
\end{align*}
such that, for all $\ov{x}_m\in \R^m$,
\begin{align}\label{fi3}
&E_i^i =  \sum_{\ontop{j\in J,}{j\not=i}}x_jp_{ij}^{ij}+\lambda_i^i,\quad i\in I,\qquad\mbox{and}\qquad
F_j^j =  -\sum_{\ontop{i\in I,}{i\not=j}}p_{ij}^{ij}x_i-\lambda_j^j,\quad j\in J,\\
& \mbox{where}\quad \lambda_k=0\quad\mbox{if}\quad k\not\in I\cap J.\nonumber
\end{align} 
We call \eqref{fi3} a {\em standard solution} of \eqref{fi1} and \eqref{fi2}. One usually wants to show that \eqref{fi1} and \eqref{fi2} have only standard solutions. Let $d$ be a positive integer. We say that $\R$ is a {\em $d$-free subset} of $\Q$ if, 
for all $m\ge 1$ and all $I,J\subseteq \{1,2,\ldots,m\}$, the following two conditions are satisfied:
\begin{enumerate}
\item[(a)] If $\max\{|I|,|J|\}\le d$, then (\ref{fi1}) implies (\ref{fi3}).
\item[(b)] If $\max\{|I|,|J|\}\le d-1$, then (\ref{fi2}) implies (\ref{fi3}).
\end{enumerate}
If (a) holds, then (b) is equivalent to the following: 
\begin{enumerate}
\item[(b$'$)] If $\max\{|I|,|J|\}\le d-1$, then (\ref{fi2}) implies (\ref{fi1}).
\end{enumerate}

The $d$-freeness of sets is usually proved by establishing a more general property called the $(t;d)$-freeness. Let us therefore define this auxiliary concept. We need to introduce some more notation. Let $t$ be a fixed element in $\Q$ such that $t\R\subseteq \R$. This assumption is not necessary for the general definition of $(t;d)$-freeness, but it will make it a bit simpler. We warn the reader that the definition that we are about to give is subject to this condition (which will be satisfied in the case we are interested in). We set $t^0=1$.
Let $a$ and $b$ be nonnegative integers, and let $E_{iu}:\R^{m-1}\to \Q$, $i\in I$, $0\le u\le a$, and $F_{jv}:\R^{m-1}\to \Q$, $j\in j$, $0\le v\le b$, be arbitrary functions. Consider the functional identity
\begin{equation}\label{5S3}
\sum_{i\in I}\sum_{u=0}^aE_{iu}^ix_it^u+\sum_{j\in J}\sum_{v=0}^bt^vx_jF_{jv}^j=0\quad\mbox{for all}\quad \ov{x}_m\in \R.
\end{equation}
Note that \eqref{5S3} coincides with \eqref{fi1} if $a=b=0$. A standard solution of \eqref{5S3} is defined as any solution of the form 
\begin{align}\label{5S4}
&E_{iu}^i =  \sum_{\ontop{j\in J,}{j\not=i}}\sum_{v=0}^bt^vx_jp_{iujv}^{ij}+\sum_{v=0}^b\lambda_{iuv}^it^v,\nonumber\\
&F_{jv}^j =  -\sum_{\ontop{i\in I,}{i\not=j}}\sum_{u=0}^ap_{iujv}^{ij}x_it^u-\sum_{u=0}^a\lambda_{juv}^jt^u,\\
& \mbox{where}\quad \lambda_{kuv}=0\quad\mbox{if}\quad k\not\in I\cap J,\nonumber
\end{align}
for all $\ov{x}_m\in \R^m$, where $p_{iujv}:\R^{m-2}\to \Q$ and $\lambda_{kuv}:\R^{m-1}\to \mathfrak{C}$ are arbitrary functions. 
We say that $\R$ is a {\em $(t;d)$-free subset} of $\Q$ if the following holds:
\begin{enumerate}
\item[(c)]If
$\max\{|I| + a,|J|+ b\}\le d$, then (\ref{5S3}) implies (\ref{5S4}).
\end{enumerate}
If $\R$ is a $(t;d)$-free subset of $\Q$ for some $t\in \Q$ (such that $t\R\subseteq \R)$, then it is also a $d$-free subset of $\Q$ --- see \cite[Lemmas 3.12 and 3.13]{FIbook}. This result is easy to prove, but it is of great importance.

\subsection{Applications to algebras of linear transformations}\label{ssALALT} 

We now return to our algebras $\cR=\fpu$, $\cS=\cR\otimes \cH$ and $\cQ= {\rm End}_\cH(U\otimes \cH)$. 
In what follows we assume that $t$ is a fixed element in $\cR$ (hence $t\cS\subseteq \cS$). The goal is to show that, under a mild condition,
$\cS$ is a $(t;d)$-free subset of $\cQ$. This condition depends on the {\em degree of algebraicity} of $t$, which we denote by $\deg(t)$. 

We can actually prove slightly more: the algebra $\cR =\fpu$ can be replaced by any larger algebra $\cA$ contained in End$(U)$. Clearly, $\cA\otimes \cH$ can be imbedded into $\cQ$ in exactly the same fashion as we have imbedded $\cR\otimes \cH$ into $\cQ$, so, by abuse of notation, we will write $\cS=\cR\otimes \cH\subset \cA\otimes \cH\subset \cQ$.

Using the lemmas of Subsection \ref{ssPR}, our goal can be achieved by making appropriate modifications in the standard proof of 
$(t;d)$-freeness of rings, based on the notion of fractional degree --- see \cite[Section 5.1]{FIbook}. Nevertheless, we will give the details for completeness. 

\begin{theorem} \label{MainT}
Let $\cA$ be any algebra such that $\fpu\subset\cA\subset\End(U)$. Let $t\in \fpu$. If $d\ge 1$ is such that $d\le \deg(t)$, then $\cA\otimes\cH$ is a $(t;d)$-free subset of $\cQ = \End_\cH(U\otimes \cH)$. In particular, $\cA\otimes\cH$ is a $d$-free subset of $\cQ$.
\end{theorem}

\begin{proof} 

One of the fundamental theorems of the theory of functional identities states that if a set is $(t;d)$-free, then any larger set is also $(t;d)$-free \cite[Theorem 3.14]{FIbook}. Therefore, we may assume without loss of generality that $\cA$ is equal to $\cR = \fpu$. Thus, all we have to show is that $\cS=\cR\otimes \cH$ is a $(t;d)$-free subset of $\cQ$, i.e., that \eqref{5S3} implies \eqref{5S4}, where $\R=\cS$, $\Q=\cQ$ and $\deg(t)\geq\max\{|I|+a,|J|+b\}$. We first consider the situation where only one summation in \eqref{5S3} appears, i.e.,  either $I=\varnothing$ or $J=\varnothing$.

\smallskip

{\bf Claim 1}. If $\deg(t)\geq |I|+a$ and 
\[
\sum_{i\in I}\sum_{u=0}^aE^i_{iu}x_it^u=0\quad\mbox{for all}\quad\ov{x}_m\in \cS^m,
\]
then all $E_{iu}=0$.

\smallskip

We prove this claim by induction on $|I|$. Suppose that $|I|=1$, say $I=\{1\}$. Then we have
\[
\sum_{u=0}^aE^1_{1u}x_1t^u = 0\quad\mbox{for all}\quad x_1\in \cS.
\]
Lemma \ref{5SL2} implies that each $E_{1u}=0$. Now suppose that $|I|>1$, say $1,2\in I$. Set
\[
H(\ov{x}_m)=\sum_{i\in I}\sum_{u=0}^aE^i_{iu}x_it^u.
\]
Then we have
\begin{align*}
0 &= H(x_1t)-H(\ov{x}_m)t\\
&= \sum_{\ontop{i\in I,}{i\not=1}}E_{i0}^i(x_1t)x_i
+\sum_{\ontop{i\in I,}{i\not=1}}\sum_{u=1}^a\left(E_{iu}^i(x_1t)-E_{i,u-1}^i\right)x_it^u
-\sum_{\ontop{i\in I,}{i\not=1}}E_{ia}^ix_it^{a+1}.
\end{align*}
We have arrived at the same type of identity as the original one, but with $I\setminus{\{1\}}$ playing the role of $I$ and $a+1$ playing the role of $a$. The induction hypothesis is therefore applicable. Consequently, $E_{ia}=0$ and $E_{iu}^i(x_1t)-E_{i,u-1}^i=0$ for all $i\not=1$, $u=0,1,\ldots,a$. Hence each $E_{iu}=0$ for $i\not=1$. Switching the roles of $x_1$ and $x_2$, we obtain $E_{1u}=0$ for every $u$.

\smallskip

{\bf Claim 2}. If $\deg(t)\geq |J|+b$ and
\[
\sum_{j\in J}\sum_{v=0}^bt^vx_jF^j_{jv}=0\quad\mbox{for all}\quad \ov{x}_m\in \cS^m, 
\]
then all $F_{jv}=0$.

\smallskip

The proof is similar to the proof of Claim 1. We are now ready to handle the general case.

\smallskip

{\bf Claim 3}. If $\deg(t)\geq\max\{|I|+a,|J|+b\}$ and \eqref{5S3} holds for all $\ov{x}_m\in \cS^m$, then \eqref{5S4} holds for all $\ov{x}_m\in \cS^m$.

\smallskip

The proof of this final claim will be longer. Let us begin by observing that it suffices to prove that the functions appearing in only one of the two summations in \eqref{5S3} have the desired form. Indeed, assuming that all the $E_{iu}$'s are of the form \eqref{5S4} (in
particular, $\lambda_{iuv} =0$ for $i\notin J$), it follows from \eqref{5S3} that
\[
\sum_{j\in J}\sum_{v=0}^bt^vx_j
\left[F_{jv}^j +\sum_{\ontop{i\in I,}{i\not=j}}\sum_{u=0}^ap_{iujv}^{ij}x_it^u+\sum_{u=0}^a\lambda_{juv}^jt^u\right]=0,
\]
and hence Claim 2 implies that all the $F_{jv}$'s are of the form \eqref{5S4} as well. Similarly, if all the $F_{jv}$'s
are of the form \eqref{5S4}, then Claim 1 implies that all the $E_{iu}$'s are also of the form \eqref{5S4}.

We proceed by induction on $|I|+|J|$. Claims 1 and 2 cover the cases where $|I|=0$ or $|J|=0$. Therefore, the first case
that has to be considered is $|I|=1=|J|$. We have to consider separately two subcases: when $I = J$ and when $I\ne J$.

In the first subcase, we may assume that $I=J=\{1\}$, so we have
\[
\sum_{u=0}^aE_{1u}^1x_1t^u+\sum_{v=0}^bt^vx_1F_{1v}^1=0.
\]
Fixing $x_2,x_3\ldots,x_m$, we may apply Lemma \ref{Lglavna} to conclude that
\[
E_{1u}(x_2,\ldots,x_m)=\sum_{v=0}^b\lambda_{1uv}(x_2,\ldots,x_m)t^v
\]
for some $\lambda_{1uv}(x_2,\ldots,x_m)\in \cC$, $0\leq u\leq a$, $0\leq v\leq b$. In other words, the $E_{1u}$'s are of the form \eqref{5S4}, and hence the $F_{1v}$'s are of the form \eqref{5S4} as well.

In the second subcase, we may assume that $I=\{2\}$ and $J=\{1\}$. Thus,
\begin{equation}\label{5ME11}
D(\ov{x}_m) \bydef \sum_{u=0}^aE_{2u}^2x_2t^u+\sum_{v=0}^bt^vx_1F_{1v}^1=0.
\end{equation}
We claim that the $F_{1v}$'s are linear in $x_2$. Indeed, if we replace $x_2$ by $\alpha x_2'+\beta x_2''$ ($\alpha,\beta\in \FF$) in \eqref{5ME11}, then we obtain
\[
\sum_{v=0}^bt^v x_1\left[F_{1v}^1(\alpha x_2'+\beta x_2'')-\alpha F_{1v}^1(x_2')-\beta F_{1v}^1(x_2'')\right]=0,
\]
where we have simplified the notation by neglecting $x_3,\ldots,x_m$. Claim 2 now implies that $F_{1v}^1(\alpha x_2'+\beta x_2'')=\alpha F_{1v}^1(x_2')+\beta F_{1v}^1(x_2'')$, as desired.  Fix $0\le u\le a$. Since $\deg(t)>a$, by Lemma \ref{Lel} there exist
$c_\ell,d_\ell\in \cR$, $\ell=1,2,\ldots,n$, such that $\sum_{\ell=1}^n c_\ell t^w d_\ell=0$ for $w\ne u$ and 
$c=\sum_{\ell=1}^n c_\ell t^ud_\ell$ is a nonzero element of $\cR$. Now \eqref{5ME11} implies
\begin{equation}\label{5S6}
0 = \sum_{\ell=1}^n D(x_2c_\ell)d_\ell = E_{2u}^2x_2c+\sum_{v=0}^bt^vx_1H_{1v}^1  
\end{equation}
where $H_{1v}^1 =H_{1v}(x_2,\ldots,x_m)=\sum_{\ell=1}^n F_{1v}(x_2c_\ell,x_3,\ldots,x_m)d_\ell$, and hence the $H_{1v}$'s are linear in $x_2$. Replacing $x_2$ by $x_2cy$ in \eqref{5S6} yields
\[
E_{2u}^2x_2cyc+\sum_{v=0}^bt^vx_1H_{1v}(x_2cy,x_3,\ldots,x_m) =0.
\]
On the other hand, \eqref{5S6} implies that 
\[
E_{2u}^2x_2cyc + \sum_{v=0}^bt^vx_1H_{1v}(x_2,x_3,\ldots,x_m)yc = 0.
\] 
Comparing the last two equations, we get
\[
\sum_{v=0}^b t^vx_1\left[H_{1v}(x_2cy,x_3,\ldots,x_m) - H_{1v}(x_2,x_3,\ldots,x_m)yc\right] = 0.
\]
Lemma \ref{5SL2} now tells us that 
\[
H_{1v}(x_2cy,x_3,\ldots,x_m) = H_{1v}(x_2,x_3,\ldots,x_m)yc
\]
for all $v=0,1,\ldots,b$ and for all $y,x_2,x_3,\ldots,x_m\in \cS$. We are now in a position to apply Lemma \ref{Lbasic}. 
Therefore, for any fixed $x_3,\ldots,x_m$ there exists $q_v(x_3,\ldots,x_m)\in \cQ$ such that
\[
H_{1v}(\ov{x}_m^1) = q_v(\ov{x}_m^{12})x_2c.
\]
Setting $p_v = -q_v$ and going back to \eqref{5S6}, we now have
\[
\left[E_{2u}^2-\sum_{v=0}^bt^vx_1p_v^{12}\right]x_2c = 0.
\]
Using Lemma \ref{5SL2} again, we conclude that $E_{2u}^2=\sum_{v=0}^bt^vx_1p_v^{12}$. This means  that all the $E_{2u}$'s are of the form 
\eqref{5S4}, and hence the same holds true for all the $F_{1v}$'s. The proof of the case $|I|=1=|J|$ is complete.

We may now assume $|I|+|J|>2$. Without loss of generality, assume $|J|\geq 2$, say $1,2\in J$. We have
\[
H(\ov{x}_m) \bydef \sum_{i\in I}\sum_{u=0}^aE_{iu}^ix_it^u+\sum_{j\in J}\sum_{v=0}^bt^vx_jF_{jv}^j = 0.
\]
Computing $H(tx_1)-tH(\ov{x}_m)$, we obtain 
\[
\begin{split}
&\sum_{i\in I}\sum_{u=0}^aG_{iu}^ix_it^u + \sum_{\ontop{j\in J,}{j\not=1}}x_jF_{j0}^j(tx_1)\\
&+\sum_{\ontop{j\in J,}{j\not=1}}\sum_{v=1}^bt^vx_j\left(F_{jv}^j(tx_1)-F_{j,v-1}^j\right)
-\sum_{\ontop{j\in J,}{j\not=1}}t^{b+1}x_jF_{jb}^j=0
\end{split}
\]
for appropriate $G_{iu}$'s. This is the same type of identity, but with $J\setminus\{1\}$ playing the role of $J$ and $b+1$ playing the role of $b$. Applying the induction hypothesis, we obtain
\begin{align*}
F_{jb}^j & =  -\sum_{\ontop{i\in I,}{i\not=j}}\sum_{u=0}^aq_{iuj\,b+1}^{ij}x_it^u-\sum_{u=0}^a\mu_{ju\,b+1}^jt^u,\quad j\not=1\\
F_{jv}^j(tx_1)-F_{j,v-1}^j & =  -\sum_{\ontop{i\in I,}{i\not=j}}\sum_{u=0}^aq_{iujv}^{ij}x_it^u-\sum_{u=0}^a\mu_{juv}^jt^u,\quad
j\not=1,\; v=1,2,\ldots,b,
\end{align*}
and $\mu_{juv}=0$ for $j\notin I$. Beginning with $F_{jb}$ and proceeding recursively, we see that the above equations yield
\[
F_{jv}^j=-\sum_{\ontop{i\in I,}{i\not=j}}\sum_{u=0}^a p_{iujv}^{ij}x_it^u-\sum_{u=0}^a\lambda_{juv}^jt^u,\quad j\not=1,
\]
for appropriate $p_{iujv}$ and $\lambda_{juv}$, with $\lambda_{juv}=0$ for $j\notin I$. In a similar fashion, by computing $H(tx_2)-tH(\ov{x}_m)$, we obtain
\[
F_{1v}^1=-\sum_{\ontop{i\in I,}{i\not=1}}\sum_{u=0}^a p_{iu1v}^{i1}x_it^u-\sum_{u=0}^a\lambda_{1uv}^1t^u,
\]
where $\lambda_{1uv}=0$ if $1\notin I$. We have shown that all the $F_{jv}$'s are of the form \eqref{5S4}. Consequently, the same holds for all the $E_{iu}$'s.
\end{proof}

If $U$ is finite-dimensional, then $\cA = \fpu = \End(U)\cong M_n(\FF)$ and hence $\cA\otimes \cH \cong M_n(\cH)$. For this algebra, our result is not new; in fact, a stronger result \cite[Corollary 2.22]{FIbook} can be proved with somewhat lesser effort. The case of infinite- dimensional $U$ is the one we are interested in. Then $\cR=\fpu$ contains elements $t$ whose degree of algebraicity is arbitrary large, so we obtain the following: 

\begin{corollary} \label{MainC}
Let $\cA$ be any algebra such that $\fpu \subset \cA\subset \End(U)$. If $U$ is infinite-dimensional, then, for any $d\ge 1$, there exists $t\in\fpu$ such that $\cA\otimes \cH$ is a  $(t;d)$-free subset of $\cQ = \End_\cH(U\otimes \cH)$. In particular, $\cA\otimes \cH$ is a $d$-free subset of $\cQ$ for any $d\ge 1$.
\end{corollary}

\subsection{Lie homomorphisms}  

We are now ready to describe Lie homomorphisms of certain Lie subalgebras of $\cA\otimes \cH$ where $\cA$ is as in Theorem \ref{MainT} and $\cH$ is a unital {\em commutative} algebra. In the case of finite-dimensional $U$, this was done in \cite{BB}. Therefore, here we will assume that $U$ is infinite-dimensional. It should be mentioned, however, that our proofs work as long as the dimension of $U$ is big enough. Our method is similar to that in \cite{BB} and relies on the results on functional identities from \cite{FIbook}.  

We say that a map $\sigma$ from an associative algebra $\cA'$ to a unital associative algebra $\cA''$ is a {\em direct sum of a homomorphism and the negative of an antihomomorphism} if there exist central idempotent $e_1$ and $e_2$ in $A''$ with $e_1+e_2=1$ such that $x\mapsto e_1\sigma(x)$ is a homomorphism and $x\mapsto e_2\sigma(x)$ is the negative of an antihomomorphism. Maps with this property are clearly Lie homomorphisms; the nontrivial question is whether all Lie homomorphisms can be described in terms of such maps.

To state our results, we will need the following notation. We set $\cA^\sharp = \cA + \FF 1$; thus, $\cA^\sharp = \cA$ if $1\in \cA$.
If $X$ is a subset of an associative algebra, then by $\langle X\rangle$ we denote the subalgebra generated by $X$.

\begin{theorem}\label{T1}
Let $\cA$ be any algebra satisfying $\fpu \subset \cA\subset \End(U)$, where $U$ is infinite-dimensional, and let $\cL$ be a noncentral Lie ideal of $\cA$. If $\cH$ is a unital commutative associative algebra, $\cM$ is a Lie ideal of some associative algebra, and $\rho\colon\cM\to \cL\otimes \cH$ is a surjective Lie homomorphism, then there exist a direct sum of a homomorphism
and the negative of an antihomomorphism $\sigma\colon\langle \cM \rangle \to \cA^\sharp\otimes \cH$ and a linear map 
$\tau\colon\cM\to 1\otimes \cH$ such that $\tau([\cM,\cM])=0$ and $\rho(x) = \sigma(x) + \tau(x)$ for all $x\in \cM$.
\end{theorem} 
 
\begin{proof}
Given any $d\ge 1$, we can find $t\in\cL$ with $\deg(t)\ge d$ (see \cite[Lemma C.5]{FIbook}). Since $\cH$ is commutative, $\cL\otimes \cH$ is a Lie ideal of $\cA\otimes\cH$. Therefore, we infer from Theorem \ref{MainT} together with \cite[Corollary 3.18]{FIbook} that $\cL\otimes\cH$ is a $d$-free subset of $\cQ$ for every $d\ge 1$.
 
Since $\cH$ is commutative, the center $\cC$ of $\cQ$ equals $1\otimes\cH$. Let $\overline{\cQ}$ be the quotient Lie algebra $\brac{\cQ}/\cC$. For any $q\in \cQ$, we set $\overline{q} = q + \cC\in \overline{\cQ}$. Define $\alpha\colon\cM\to\overline{\cQ}$ by $\alpha(x) = \overline{\rho(x)}$. Then $\alpha$ is a Lie homomorphism whose range is equal to $\ov{\cL\otimes \cH}$. Since  $\cL\otimes \cH$ is $d$-free for every $d$, in particular for $d=7$, we may use \cite[Theorem 6.19]{FIbook} to conclude that there exists a direct sum of a homomorphism and the negative of an antihomomorphism $\sigma\colon\langle \cM \rangle \to \cQ$ such that $\alpha(x) = \overline{\sigma(x)}$ for all $x\in \cM$, i.e., $\overline{\rho(x)} =  \overline{\sigma(x)}$ for all $x\in \cM$. Hence $\tau(x) \bydef \sigma(x)-\rho(x)$ lies in $\cC$. Therefore, $\sigma(\cM)\subset (\cL + \FF 1)\otimes \cH$. This clearly implies that $\sigma$ maps $\langle \cM\rangle$ into $\cA^\sharp\otimes \cH$. Finally, both $\rho$ and $\sigma$ are Lie homomorphisms, so 
\begin{align*}
\tau([x,y])& = \sigma([x,y]) - \rho([x,y]) = [\sigma(x),\sigma(y)] - [\rho(x),\rho(y)]\\
& = [\rho(x)+\tau(x),\rho(y)+\tau(y)] - [\rho(x),\rho(y)] = 0
\end{align*}
for all $x,y\in \cM$. 
\end{proof}

If $\cA$ is any associative algebra with an involution $\vphi$, then by $\sks(\cA,\vphi)$ we denote the Lie algebra of skew-symmetric elements in $\cA$: 
\[
\sks(\cA,\vphi)=\{a\in\cA\;|\;\vphi(a) = -a\}.
\]

\begin{theorem}\label{T2}
Let $\cA$ be any algebra satisfying $\fpu \subset \cA\subset \End(U)$, where $U$ is infinite-dimensional and the characteristic of the ground field is different from $2$. Assume that $\cA$ has an involution $\vphi$ and let $\cL$ be a noncentral Lie ideal of $\sks(\cA,\vphi)$. If $\cH$ is a unital commutative associative algebra, $\cU$ is an associative algebra with involution $*$, $\cM$ is a Lie ideal of $\sks(\cU,*)$, and $\rho\colon\cM\to\cL\otimes \cH$ is a surjective Lie homomorphism, then there exist a homomorphism
$\sigma\colon\langle \cM \rangle \to \cA^\sharp\otimes \cH$ and a linear map $\tau\colon \cM\to 1\otimes \cH$ such that $\tau([\cM,\cM])=0$ and $\rho(x) = \sigma(x) + \tau(x)$ for all $x\in \cM$.
\end{theorem}

\begin{proof}
Extend $\vphi$ to $\cA\otimes \cH$ by setting $\vphi(x\otimes h) = \vphi(x)\otimes h$. Then $\cA\otimes \cH$ is a Lie ideal of $\sks(\cA\otimes \cH,\vphi) = \sks(\cA,\vphi)\otimes \cH$.
 
Given any $d\ge 1$, we can find $t\in \cL$ with $\deg(t)\ge 2d+3$ (see \cite[Lemma C.6]{FIbook}). Theorem \ref{MainT} tells us that $\cA\otimes \cH$ is a $(t;2d+3)$-free subset of $\cQ$. Now we infer from \cite[Theorem 3.28]{FIbook} that $\sks(\cA\otimes \cH,\vphi)$ is a $(t;d+1)$-free subset of $\cQ$. Using \cite[Corollary 3.18]{FIbook}, we conclude that $\cL\otimes \cH$ is a $(t;d)$-free subset of $\cQ$ and hence $d$-free. 

Let $\overline{\cQ}$ and $\overline{q}$ have the same meaning as in the proof of Theorem \ref{T1}.
Define $\alpha\colon\cM\to \overline{\cQ}$ by $\alpha(x) = \overline{\rho(x)}$. Since $\cL\otimes \cH$ is $d$-free for every $d$, in particular for $d=9$, and $\alpha(\cM) = \overline{\cL\otimes \cH}$, we can use \cite[Theorem 6.18]{FIbook} to conclude that there is a homomorphism $\sigma\colon\langle \cM \rangle \to \cQ$ such that $\alpha(x) = \overline{\sigma(x)}$ for all $x\in \cM$. Then the argument from the end of the proof of Theorem \ref{T1} shows that $\tau= \rho - \sigma$ is a map from $\cM$ into $1\otimes \cH$ vanishing on $[\cM,\cM]$ and that the range of $\sigma$ lies in $\cA^\sharp\otimes\cH$.
\end{proof}

%-------------------------------------------------------------------------------

\section{Gradings on Lie algebras of finitary linear transformations}\label{sLie}

Throughout this section, $\FF$ is an algebraically closed field of characteristic different from $2$, and $G$ is an abelian group. Our goal  is to classify $G$-gradings on the infinite-dimensional simple Lie algebras $\fsl(U,\Pi)$, $\fso(U,\Phi)$ and $\fsp(U,\Phi)$ of finitary linear transformation over $\FF$. We are going to transfer the classification results for the associative algebras $\fpu$ in Section \ref{sSTGPA} to the above Lie algebras using the results on Lie homomorphisms in Section \ref{sFILH}. 

We denote by $\cH$ the group algebra $\FF G$, which is a commutative and cocommutative Hopf algebra. A $G$-grading on an algebra $\cU$ is equivalent to a (right) $\cH$-comodule structure $\rho\colon\cU\to\cU\ot\cH$ that is also a homomorphism of algebras. Indeed, given a $G$-grading on $\cU$, we set $\rho(x)=x\ot g$ for all $x\in\cL_g$ and extend by linearity. Conversely, given $\rho$, we obtain a $G$-grading on $\cU$ by setting $\cU_g=\{x\in\cU\;|\;\rho(x)=x\ot g\}$. We can extend $\rho$ to a homomorphism $\cU\ot\cH\to\cU\ot\cH$ by setting $\rho(x\ot h)=\rho(x)(1\ot h)$. It is easy to see that this extended homomorphism is surjective.

\subsection{Special linear Lie algebras}\label{ssSLA}

Let $\cL=\fsl(U,\Pi)$ where $U$ is infinite-di\-men\-sional and let $\cR=\fpu$. Suppose $\cL$ is given a $G$-grading and let $\rho\colon\cL\ot\cH\to\cL\ot\cH$ be the Lie homomorphism obtained by extending the comodule structure map. Then $\cL=[\cR,\cR]$ is a noncentral Lie ideal of $\cR$ and $\cM\bydef\cL\ot\cH$ is a Lie ideal of $\cR\ot\cH$. Moreover, $\langle \cL\rangle=\cR$ implies $\langle \cM\rangle=\cR\ot\cH$, and $[\cL,\cL]=\cL$ implies $[\cM,\cM]=\cM$. Applying Theorem \ref{T1} and observing that $\tau=0$, we conclude that $\rho$ extends to a map $\rho'\colon\cR\ot\cH\to\cR\ot\cH$ which is a sum of a homomorphism and the negative of an antihomomorphism. Thus, there are central idempotents $e_1$ and $e_2$ in $\cH$ (which can be identified with the center of $\cR^\sharp\ot\cH$) with $e_1+e_2=1$ such that the composition of $\rho'$ and the projection $\cR\ot\cH\to\cR\ot e_1\cH$ is a homomorphism, while the composition of $\rho'$ and the projection $\cR\ot\cH\to\cR\ot e_2\cH$ is the negative of an antihomomorphism. 

If $\psi$ is any automorphism of $\cL$, it can be extended to a map $\psi'\colon\cR\to\cR$ that is a homomorphism or the negative of an antihomomorphism (use \cite[Theorem 6.19]{FIbook} or our Theorem \ref{T1} with $\cH=\FF$). Clearly, $\psi'$ is surjective. Since $\cR$ is simple, we conclude that $\psi'$ is an automorphism or the negative of an antiautomorphism. We claim that $\psi$ cannot admit extensions of both types. Indeed, assume that $\psi'$ is an automorphism of $\cR$ and $\psi''$ is the negative of an antiautomorphism of $\cR$ such that both restrict to $\psi$. Let $\sigma=(\psi')^{-1}\psi''$. Then $\sigma$ is the negative of an antiautomorphism of $\cR$ that restricts to identity on $\cL$. Hence, for any $x\in\cL$ and $r\in\cR$, we have $[x,r]=\sigma([x,r])=[\sigma(x),\sigma(r)]=[x,\sigma(r)]$. It follows that $r-\sigma(r)$ belongs to the center of $\cR$, which is zero, so $\sigma$ is the identity map --- a contradiction. We have shown that, for any automorphism $\psi$ of $\cL$, there is a unique extension $\psi'\colon\cR\to\cR$ that is either an antiautomorphism or the negative of an antiautomorphism.

Now any character $\chi\in\wh{G}$ acts as an automorphism $\psi$ of $\cL$ defined by $\psi(x)=\chi(g)x$ for all $x\in\cL_g$. Denote this $\psi$ by $\eta(\chi)$, i.e., $\eta(\chi)=(\id_\cR\ot\chi)\rho$. Clearly, $\eta$ is a homomorphism from $\wh{G}$ to $\Aut(\cL)$. Define $\eta'(\chi)=\eta(\chi)'$. It follows from the uniqueness of extension that $\eta'$ is a homomorphism from $\wh{G}$ to the group $\overline{\Aut}(\cR)$ consisting of all automorphisms and the negatives of antiautomorphisms of $\cR$. We can regard $\chi$ as a homomorphism of algebras $\cH\to\FF$ and define $\eta''(\chi)=(\id_\cR\ot\chi)\rho'$. Then $\eta''(\chi)$ is a map $\cR\to\cR$ that restricts to $\eta(\chi)$ on $\cL$. Clearly, $\chi(e_1)$ is either $1$ or $0$, and then $\chi(e_2)$ is either $0$ or $1$, respectively. An easy calculation shows that if $\chi(e_1)=1$, then $\eta''(\chi)$ is a homomorphism, and if $\chi(e_1)=0$, then $\eta''(\chi)$ is the negative of an antihomomorphism. We conclude that $\eta''(\chi)=\eta'(\chi)$. Setting $h=e_1-e_2$, we see that, for any $\chi\in\wh{G}$,
\begin{equation}\label{eq:characterize_h}
\chi(h)=\left\{\begin{array}{ll}
\phantom{+}1 & \mbox{if}\quad \eta'(\chi)\in\Aut(\cR);\\
-1 & \mbox{if}\quad \eta'(\chi)\notin\Aut(\cR).
\end{array}\right.
\end{equation}
It is well known that any idempotent of $\cH=\FF G$ is contained in $\FF K$ for some finite subgroup $K\subset G$ such that $\chr{\FF}$ does not divide the order of $K$. Pick $K$ such that $e_1\in\FF G_0$, then also $h\in\FF K$. Since any character of $K$ extends to a character of $G$, equation \eqref{eq:characterize_h} implies that $(\chi_1\chi_2)(h)=\chi_1(h)\chi_2(h)$ for all $\chi_1,\chi_2\in\wh{K}$, so $h$ is a group-like element of $(\FF\wh{K})^*=\FF K$. It follows that $h\in K$. Clearly, the order of $h$ is at most $2$. Since $\chr{\FF}\ne 2$, the characterization given by \eqref{eq:characterize_h} shows that the element $h$ is uniquely determined by the given $G$-grading on $\cL$. It follows that $e_1$, $e_2$ and $\rho'$ are also uniquely determined.

If $h$ has order $1$, then $e_1=1$ and $e_2=0$, so the map $\rho'\colon\cR\to\cR\ot\cH$ is a homomorphism of associative algebras. Since $\cL$ generates $\cR$, it immediately follows that $\rho'$ is a comodule structure. This makes $\cR$ a $G$-graded algebra such that the given grading on $\cL$ is just the restriction of the grading on $\cR$, i.e., $\cL_g=\cR_g\cap\cL$ for all $g\in G$. Such gradings on $\cL$ will be referred to as {\em grading of Type I}. We can use Corollary \ref{lfd_abelian_G} to obtain all such gradings.

If $h$ has order $2$, then both $e_1$ and $e_2$ are nontrivial, so $\rho'$ is not a homomorphism of associative algebras, which means that the algebra $\cR$ does not admit a $G$-grading that would restrict to the given grading on $\cL$. Such gradings on $\cL$ will be referred to as {\em grading of Type II (with distinguished element $h$)}. Let $\bG=G/\langle h\rangle$ and $\overline{\cH}=\FF\bG$. Denote the quotient map $G\to\bG$ by $\pi$ and extend it to a homomorphism of Hopf algebras $\cH\to\overline{\cH}$. Since $\pi(e_2)=0$, the map $\overline{\rho}\bydef(\id_\cR\ot\pi)\rho$ is a homomorphism of associative algebras, so $\overline{\rho}\colon\cR\to\cR\ot\overline{\cH}$ is a comodule structure, which makes $\cR$ a $\bG$-graded algebra. The restriction of this grading to $\cL$ is the coarsening of the given $G$-grading induced by $\pi\colon G\to\bG$, i.e., $\cR_\bg\cap\cL=\cL_g\oplus\cL_{gh}$ for all $g\in G$, where $\bg=\pi(g)$. The original grading can be recovered as follows.

Fix a character $\chi\in\wh{G}$ satisfying $\chi(h)=-1$ and let $\psi=\eta(\chi)$. Then we get
\[
\cL_g=\{x\in\cL_\bg\;|\;\psi(x)=\chi(g)x\}.
\]
Indeed, by definition of $\psi$, $\cL_g$ is contained in the right-hand side. Conversely, suppose $x\in\cL_\bg$ satisfies $\psi(x)=\chi(g)x$. Write $x=y+z$ where $y\in\cL_g$ and $y\in\cL_{gh}$. Then we have $\chi(g)(y+z)=\psi(x)=\chi(g)(y-z)$. It follows that  $z=0$ and so $x\in\cL_g$. 

By \eqref{eq:characterize_h}, the extension $\psi'$ of $\psi$ is not an automorphism, so $\psi'=-\vphi$ where $\vphi$ is an antiautomorphism of $\cR$. Since $\psi$ leaves the components $\cL_\bg$ invariant and $\cL$ generates $\cR$, it follows that $\vphi$ leaves the components $\cR_\bg$ invariant. Since $\vphi^2=(\psi')^2=\eta'(\chi^2)$ and $\chi^2$ can be regarded as a character of $\bG$, we obtain $\vphi^2(r)=\chi^2(\bg)r$ for all $r\in\cR_\bg$. Set
\begin{equation}\label{eqant}
\cR_g=\{r\in\cR_{\bg}\;|\;\vphi(r)=-\chi(g)r\}.
\end{equation}
Then $\cR_\bg=\cR_g\oplus\cR_{gh}$ and hence $\cR=\bigoplus_{g\in G}\cR_g$. Since $\psi'$ is not an automorphism of $\cR$, this is not a $G$-grading on the associative algebra $\cR$. It is, however, a $G$-grading on the Lie algebra $\brac{\cR}$ (corresponding to the comodule structure $\rho'$). The $\bG$-grading on $\cR$ and the antiautomorphism $\vphi$ completely determine the $G$-grading on $\cL$:
\begin{equation}\label{eq:grading_type_II}
\cL_g=\{x\in\cR_\bg\cap\cL\;|\;\vphi(x)=-\chi(g)x\}.
\end{equation}

Conversely, suppose we have a $\bG$-grading on $\cR$ and an antiautomorphism $\vphi$ of the $\bG$-graded algebra $\cR$ satisfying the following {\em compatibility condition}:
\begin{equation}\label{eq:chi_compatibility}
\vphi^2(r)=\chi^2(\bg)r\quad\mbox{for all}\quad r\in\cR_{\bg},\, \bg\in\bG. 
\end{equation}
Since $-\vphi$ is an automorphism of $\brac{\cR}$, equation \eqref{eqant} gives a $G$-grading on $\brac{\cR}$ that refines the given $\bG$-grading. Since $\cL=[\cR,\cR]$ is a $G$-graded subalgebra of $\brac{\cR}$, we see that \eqref{eq:grading_type_II} defines a Type II grading on $\cL$ with distinguished element $h$. Note that the compatibility condition implies that $\vphi$ acts as involution on the identity component of the $\bG$-grading of $\cR$, so Theorem \ref{lfd_2} tells us that $(\cR,\vphi)$ is isomorphic to some $\frF(\bG,\cD,\tV,\tB,\bg_0)$.

\begin{proposition}\label{lfd_chi}
The graded algebra $\cR=\frF(\bG,\cD,\tV,\tB,\bg_0)$ with its antiautomorphism $\vphi$ satisfies the compatibility condition \eqref{eq:chi_compatibility} if and only if $\pi\colon G\to\bG$ splits over the support $\bT$ of $\cD$ and there exists $\mu_0\in\FF^\times$ such that $\tB$ satisfies the symmetry condition \eqref{eq:def_mu} where $\mu_A$, for all $A\in\bG/\bT$, is given by
\begin{equation}\label{eq:mu_for_chi}
\mu_A=\left\{
\begin{array}{ll}
\mu_0\chi^{-2}(A)\beta(\tau(A)) & \mbox{if}\quad \bg_0 A^2=\bT;\\
\mu_0\chi^{-2}(A) & \mbox{if}\quad \bg_0 A^2\ne \bT;
\end{array}\right.
\end{equation}
where we regard $\chi^2$ as a character of $\bG/\bT$ (since $\chi^2$ is trivial on $\bT$).
\end{proposition}

\begin{proof}
Equation \eqref{eq:vphi_square} tells us that, in terms of Theorem \ref{isomorphism_graded_simple}, $\vphi^2$ is determined by the isomorphism $(\id,Q^{-1},Q^*)$ from $(\cD,V,W)$ to itself. On the other hand, the action of $\chi^2$ on $\cR$ is determined by the isomorphism $(\psi_0,\psi_1,\psi_2)$ from $(\cD,V,W)$ to itself, where $\psi_0$, $\psi_1$ and $\psi_2$ are the actions of $\chi^2$ on $\cD$, $V$ and $W$, respectively, arising from their $\bG$-gradings. We conclude that \eqref{eq:chi_compatibility} holds if and only if $\psi_0=\id$ and $\psi_1=\mu_0 Q^{-1}$ for some $\mu_0\in\FF^\times$. The first condition means that $\chi^2(\bt)=1$ for all $\bt\in\bT$, which is equivalent to $H\bydef\pi^{-1}(\bT)$ being isomorphic to $\bT\times\langle h\rangle$. The second condition means that $\lambda_{A}\chi^2(A)=\mu_0$ for all $A$, which is equivalent to \eqref{eq:mu_for_chi} in view of \eqref{eq:lambda_and_mu}.
\end{proof}

We note that, since $\mu_A\mu_{\bg_0^{-1}A^{-1}}=1$ for all $A$, the scalar $\mu_0$ satisfies
\[
\mu_0^2=\chi^{-2}(\bg_0)
\]
and hence can take only two values.

To state the classification of $G$-gradings on the Lie algebras $\fsl(U,\Pi)$, we introduce the model $G$-graded algebras $\frA^\mathrm{(I)}(G,T,\beta,\tV,\tW)$ and $\frA^{\mathrm{(II)}}(G,H,h,\beta,\tV,\tB,\bg_0)$. 

Let $T\subset G$ be a finite subgroup with a nondegenerate alternating bicharacter $\beta$. The Lie algebra $\frA^\mathrm{(I)}(G,T,\beta,\tV,\tW)$ is just the commutator subalgebra of the $G$-graded associative algebra $\cR=\mathfrak{F}(G,T,\beta,\tV,\tW)$ (introduced before Corollary \ref{lfd_abelian_G}). By Theorem \ref{lfd}, it is isomorphic to $\fsl(U,\Pi)$ where $U=\tV^\ell$, $\Pi=\tW^\ell$, $\ell^2=|T|$, and the bilinear form $\Pi\times U\to\FF$ is given by \eqref{eq:two_bilinear_forms}.

Let $H\subset G$ be a finite elementary $2$-subgroup, $h\ne e$ an element of $H$, and $\beta$ a nondegenerate alternating bicharacter on $H/\langle h\rangle$. Fix a charcter $\chi\in\wh{G}$ with $\chi(h)=-1$. Let $\bG=G/\langle h\rangle$ and $\bT=H/\langle h\rangle$. The Lie algebra $\frA^{\mathrm{(II)}}(G,H,h,\beta,\tV,\tB,\bg_0)$ is the commutator subalgebra of the Lie algebra $\brac{\cR}$ with a $G$-grading defined by refining the $\bG$-grading as in \eqref{eqant} where $\cR$ is the $\bG$-graded associative algebra $\mathfrak{F}(\bG,\bT,\beta,\tV,\tB,\bg_0)$ with antiautomorphism $\vphi$ (introduced before Theorem \ref{lfd_2}) and the bilinear form $\tB$ on $\tV$ satisfies the symmetry condition \eqref{eq:def_mu} with $\mu_A$ given by \eqref{eq:mu_for_chi} for some $\mu_0\in\FF^\times$. By Theorem \ref{lfd_2}, $\frA^{\mathrm{(II)}}(G,H,h,\beta,\tV,\tB,\bg_0)$ is isomorphic to $\fsl_U(U)$ where $U=\tV^\ell$, $\ell^2=|H|/2$, and the bilinear form $U\times U\to\FF$ is given by \eqref{eq:two_bilinear_forms_2}.

The definition of $\frA^{\mathrm{(II)}}(G,H,h,\beta,\tV,\tB,\bg_0)$ depends on the choice of $\chi$. However, regardless of this choice, we obtain the same collection of graded algebras as $(\tV,\tB)$ ranges over all possibilities allowed by the chosen $\chi$. We assume that a choice of $\chi$ is fixed for any element $h\in G$ of order $2$.

Now we can state our main result about gradings on special Lie algebras of finitary linear operators on an infinite-dimensional vector space.

\begin{theorem}\label{tmsla}
Let $G$ be an abelian group and let $\FF$ be an algebraically closed field, $\chr{\FF}\ne 2$. If a special Lie algebra $\cL$ of finitary linear operators on an infinite-dimensional vector space over $\FF$ is given a $G$-grading, then $\cL$ is isomorphic as a graded algebra to some $\frA^\mathrm{(I)}(G,T,\beta,\tV,\tW)$ or $\frA^{\mathrm{(II)}}(G,H,h,\beta,\tV,\tB,\bg_0)$.
No $G$-graded Lie algebra with superscript \emph{(I)} is isomorphic to one with superscript \emph{(II)}. Moreover,
\begin{itemize}
\item $\frA^\mathrm{(I)}(G,T,\beta,\tV,\tW)$ and $\frA^\mathrm{(I)}(G,T',\beta',\tV',\tW')$ are isomorphic if and only if $T'=T$ and either  $\beta'=\beta$ and $(\tV',\tW')\sim(\tV,\tW)$, or $\beta'=\beta^{-1}$ and $(\tV',\tW')\sim(\tW,\tV)$ as in Definition \ref{df:datum_finitary_ab}.
\item $\frA^{\mathrm{(II)}}(G,H,h,\beta,\tV,\tB,\bg_0)$ and $\frA^{\mathrm{(II)}}(G,H',h',\beta',\tV',\tB',\bg'_0)$ are isomorphic if and only if $H'=H$, $h'=h$, $\beta'=\beta$, and $(\tV',\tB',\bg'_0)\sim(\tV,\tB,\bg_0)$ as in Definition \ref{df:datum_finitary_2}.
\end{itemize}
\end{theorem}

\begin{proof}
The first assertion is already proved. Suppose that $\cL$ and $\cL'$ are $G$-graded Lie algebras of the above form. If we have an isomorphism $\cL\to\cL'$ of graded algebras, then it extends to an isomorphism or the negative of an antiisomorphism $\cR\to\cR'$ of the corresponding associative algebras, so we obtain an isomorphism $\psi\colon\cR\to\cR'$ or $\psi\colon\cR^{\mathrm{op}}\to\cR'$. If $\cL$ is of Type I, then $\cR$ has a $G$-grading that restricts to the given grading on $\cL$. Clearly, $\psi$ sends this $G$-grading on $\cR$ to a $G$-grading on $\cR'$ that restricts to the given grading on $\cL'$, hence $\cL'$ is also of Type I. Since the gradings on $\cR$ and $\cR'$ are uniquely determined by their restrictions to $\cL$ and $\cL'$, respectively, we conclude that $\psi$ is an isomorphism of graded algebras. Applying Corollary \ref{lfd_abelian_G} and observing that passing from $\cR$ to $\cR^{\mathrm{op}}$ is equivalent to  replacing $\cD$ by $\cD^{op}$ and switching $\tV$ and $\tW$ (see the discussion at the beginning of Subsection \ref{ssASF}), we get the condition in the first item above. The sufficiency of this condition is obvious.

It remains to consider the case where both $\cL$ and $\cL'$ are of Type II. Since the action of any character of $G$ extends uniquely from $\cL$ to $\cR$ (respectively, from $\cL'$ to $\cR'$), $\psi$ preserves the $\wh{G}$-action, so the characterization of $h$ given by \eqref{eq:characterize_h} shows that $h=h'$. Hence both $\cR$ and $\cR'$ have $\bG$-gradings, where $\bG=G/\langle h\rangle$, and $\psi$ is an isomorphism of $\bG$-graded algebras. Moreover, since the antiautomorphisms $\vphi$ and $\vphi'$ are determined by the action of the same character, we have $\vphi'=\psi\circ\vphi\circ\psi^{-1}$. Note that $(\cR,\vphi)$ is isomorphic to $(\cR^{\mathrm{op}},\vphi)$ by virtue of $\vphi$ itself. By Theorem \ref{lfd_2}, it follows that $H/\langle h\rangle=H'/\langle h\rangle$, $\beta=\beta'$ and $(\tV',\tB',\bg'_0)\sim(\tV,\tB,\bg_0)$, which yields the condition in the second item above. The sufficiency of this condition is obvious.
\end{proof}

\subsubsection{Graded bases} 

The Lie algebra $\cL=\frA^\mathrm{(I)}(G,T,\beta,\tV,\tW)$, is the commutator subalgebra of $\cR=\mathfrak{F}(G,T,\beta,\tV,\tW)$. As a vector space over $\FF$, this latter can be written as $\tV\ot\cD\ot\tW$ --- see Subsection \ref{ssGSAMI}. Recall that, for $A,A'\in G/T$, we have $(\tW_A,\tV_{A'})=0$ unless $A'=A^{-1}$, and $\tW_{A^{-1}}$ is a total subspace of  $\tV_A^\ast$. The action of the tensor $v\ot d\ot w\in\tV_A\ot\cD\ot\tW_{(A')^{-1}}$ on $u\ot d'\in\tV_{A''}\ot\cD$ is given by $(v\ot d\ot w)(u\ot d')=v\ot d(w,u)d'=(w,u)(v\ot (dd'))$, which is zero unless $A''=A'$. Also, the degree of the element $v\ot d\ot w$ in the $G$-grading equals $\gamma(A)(\deg d)\gamma(A')^{-1}$. Computing the commutators of such elements and using our standard notation $X_t$ for the basis elements of $\cD$, we find that $\cL$ is spanned by the elements of the form $v\ot X_t\ot w$ where $(w,v)=0$ or $t=e$. Note that since $\tW$ is a total subspace of $\tV^\ast$, there exist $A_0\in G/T$, $v_0\in \tV_{A_0}$ and $w_0\in\tW_{A_0^{-1}}$ such that $(w_0,v_0)=1$. Then $\FF(v_0\ot 1\ot w_0)\oplus\cL=\cR$. Given bases for the vector spaces $\tV_A$ and $\tW_A$ for all $A\in G/T$  such that the basis for $\tV_{A_0}$ includes $v_0$ and the basis for $\tW_{A_0^{-1}}$ includes $w_0$, we obtain a graded basis for $\cL=\frA^\mathrm{(I)}(G,T,\beta,\tV,\tW)$ consisting of the following two sets: first, the elements of the form $v\ot 1\ot w-(w,v)(v_0\ot 1\ot w_0)$ with $v\ne v_0$ or $w\ne w_0$, and, second, the elements of the form $v\ot X_t\ot w$ with $t\ne e$, where $v$ and $w$ range over the given bases of $\tV$ and $\tW$, respectively. Hence, for $g\ne e$, a basis for $\cL_g$ consists of the elements
\[
E^{\mathrm{(I)}}_{g,v,w}\bydef v\ot X_t\ot w
\]
where $v$ ranges over a basis of $\tV_A$, $w$ ranges over a basis of $\tW_{gA^{-1}}$, and $A$ ranges over $G/T$, while $t=g\gamma(A)^{-1}\gamma(g^{-1}A)$. For $g=e$, we take the elements
\[
E^{\mathrm{(I)}}_{g,v,w}-(w,v)E^{\mathrm{(I)}}_{g,v_0,w_0}
\]
and discard the zero obtained when $v=v_0$ and $w=w_0$.

In the case of $\cL=\frA^{\mathrm{(II)}}(G,H,h,\beta,\tV,\tB,\bg_0)$, we have to start with the $\bG$-grading on  $\cR=\mathfrak{F}(\bG,\bT,\beta,\tV,\tB,\bg_0)$ where $\bG=G/\langle h\rangle$ and $\bT=H/\langle h\rangle$. Now $\cR$ can be written as $\tV\ot\cD\ot\tV$ where the action of the tensor $v\ot d\ot w\in\tV_A\ot\cD\ot\tV_{\bg_0^{-1}(A')^{-1}}$ on $u\ot d'\in\tV_{A''}\ot\cD$ is given by $(v\ot d\ot w)(u\ot d')=v\ot dB(w,u)d'$, which is zero unless $A''=A'$. Here we are using the notation of Subsection \ref{ssAAIIC} --- in particular, the $\vphi_0$-sesquilinear form $B$ given by \eqref{eq:B_and_tB_1} and \eqref{eq:B_and_tB_2} --- and we have identified $w\in\tV$ with $f_w\in W$ given by $f_w(u)=B(w,u)$ for all $u\in V$. 
The degree of the element $v\ot d\ot w$ in the $\bG$-grading equals 
$\gamma(A)(\deg d)\bg_0\gamma(\bg_0^{-1}(A')^{-1})=\gamma(A)(\deg d)\tau(A')\gamma(A')^{-1}$ where, as before,  $\tau(A')=\bg_0\gamma(A')^2$ if $\bg_0(A')^2=\bT$, and we have set $\tau(A')=\ov{e}$ if $\bg_0(A')^2\ne\bT$. Also, \eqref{eq:vphi_on_CF} yields $\vp(v\ot d\ot w)=Q^{-1}w\ot\vp_0(d) v$. 
Taking $d=X_t$ and recalling that $\vp_0(X_t)=\beta(t)X_t$, we obtain 
\[
\vp(v\ot X_t\ot w)=\beta(t)(Q^{-1}(w)\ot X_t\ot v).
\]

With our fixed character $\chi\colon G\to\FF^\times$ satisfying $\chi(h)=-1$, the $G$-grading on the vector space $\cR$ is given by \eqref{eqant}. 
Taking into account that $\vp^2(r)=\chi(g)^2 r$ for any $r\in\cR_\bg$, we can write
\[
r=\frac{1}{2}\left(r-\frac{1}{\chi(g)}\vp(r)\right)+\frac{1}{2}\left(r+\frac{1}{\chi(g)}\vp(r)\right)
\]
where
\[
\frac{1}{2}\left(r-\frac{1}{\chi(g)}\vp(r)\right)\in \cR_g\quad\mbox{and}\quad
\frac{1}{2}\left(r+\frac{1}{\chi(g)}\vp(r)\right)\in \cR_{gh}.
\]
Since $\chi(gh)=-\chi(g)$, the second expression is identical to the first one if $g$ is replaced by $gh$.
Substituting $r=v\ot X_t\ot w$ and the above expression for $\vp(v\ot X_t\ot w)$, we find that $\cR_g$ is spanned by the elements
\begin{equation*}%\label{echiQ}
v\ot X_t\ot w-\frac{\beta(t)}{\chi(g)}Q^{-1}w\ot X_t\ot v,
\end{equation*}
where $v\in\tV_A$, $w\in\tV_{\bg_0^{-1}(A')^{-1}}$, and $\gamma(A)t\tau(A')\gamma(A')^{-1}=\bg$.
The eigenvalues of $Q$ are given by \eqref{eq:lambda_and_mu}, and $\mu_{\bg_0^{-1}(A')^{-1}}=\mu_{A'}^{-1}$, so we can rewrite the above spanning elements as follows:
\begin{equation*}%\label{echiLambda1}
%\left\{
%\begin{array}{ll}
v\ot X_t\ot w-\mu_{A'}\beta(t)\beta(\tau(A'))\chi^{-1}(g)(w\ot X_t\ot v).% & \mbox{if}\quad \bg_0 (A')^2=\bT;\\
%v\ot X_t\ot w-\mu_{A'}\beta(t)\chi^{-1}(g)(w\ot X_t\ot  v)& \mbox{if}\quad \bg_0 (A')^2\ne \bT.
%\end{array}\right.
\end{equation*}
Finally, recalling that $\mu_{A'}$ is determined in Proposition \ref{lfd_chi} and $A'=\bg^{-1}A$, we conclude that $\cR_g$ is spanned by the elements 
\begin{equation}\label{eq:basis_AII}
E^{\mathrm{(II)}}_{g,v,w}\bydef v\ot X_t\ot w-\dfrac{\mu_0\beta(t)\chi(g)}{\chi^2(A)}(w\ot X_t\ot v)
\end{equation}
where $v$ ranges over a basis of $\tV_A$, $w$ ranges over a basis of $\tV_{\bg_0^{-1}\bg A^{-1}}$, and $A$ ranges over $\bG/\bT=G/H$, while $t=\bg\gamma(A)^{-1}\gamma(\bg^{-1}A)\tau(\bg^{-1}A)$. If we discard zeros among the elements $E^{\mathrm{(II)}}_{g,v,w}$, the remaining ones form a basis for $\cR_g$. This is also a basis for $\cL_g$ unless $g=e$ or $g=h$. For these two cases, bases can be obtained using the same idea as for Type I, namely, subtracting a suitable scalar multiple of $E^{\mathrm{(II)}}_{g,v_0,w_0}$, where  $v_0\in\tV_{A_0}$ and $w_0\in\tV_{\bg_0^{-1}A_0^{-1}}$ are such that $\tB(w_0,v_0)=1$. Specifically, for $g=e$ and $g=h$, we replace the elements \eqref{eq:basis_AII} by
\[
E^{\mathrm{(II)}}_{g,v,w}-\beta(t_0)\beta(t)\tB(w,v)E^{\mathrm{(II)}}_{g,v_0,w_0}
\]
where $t=\tau(A)$ and $t_0=\tau(A_0)$.

\subsubsection{Countable case} 

In the case $\cL=\Sl(\infty)$, we can express the classification of $G$-gradings in combinatorial terms. Here $\cR=M_\infty(\FF)$, whose $G$-gradings are classified in Corollary \ref{gradings_on_M_infinity}. Namely, $\cR$ is isomorphic to $\frF(G,\cD,\kappa)$, which we can write as $\frF(G,T,\beta,\kappa)$ because $G$ is abelian. Recall that the function $\kappa\colon G/T\to\{0,1,2,\ldots,\infty\}$ has a finite or countable support; here $|\kappa|\bydef\sum_{A\in G/T}\kappa(A)$ must be infinite. The $G$-grading on $\frF(G,T,\beta,\kappa)$ restricts to $\cL$, and we denote the resulting $G$-graded Lie algebra by $\frA^\mathrm{(I)}(G,T,\beta,\kappa)$. In this way we obtain all Type I gradings on $\cL$. 

For Type II, we can use Corollary \ref{gradings_on_M_infinity_2}, with $G$ replaced by $\bG=G/\langle h\rangle$ and with the function $\mu\colon A\mapsto\mu_A$ ($A\in\bG/\bT$) satisfying  \eqref{eq:mu_for_chi}. With the other parameters fixed, there are at most two such functions. Indeed, if there is $A_0\in\bG/\bT$ such that $\bg_0 A_0^2=\bT$ and $\kappa(A_0)$ is finite and odd (forcing $\mu_{A_0}=1$), then there is at most one admissible function $\mu$, defined by \eqref{eq:mu_for_chi} with $\mu_0=\chi^2(A_0)\beta(\tau(A_0))$. Such function exists if and only if all $A_0$ of this kind produce the same value $\mu_0$. If there is no $A_0$ of this kind, then there are exactly two admissible functions $\mu$, defined by \eqref{eq:mu_for_chi} where $\mu_0$ satisfies $\mu_0^2=\chi^{-2}(\bg_0)$. Denote by $\frA^\mathrm{(II)}(G,H,h,\beta,\kappa,\mu_0,\bg_0)$ the $G$-graded Lie algebra obtained by restricting to $\cL$ the refinement \eqref{eqant} of the $\bG$-grading on $\frF(G,\bT,\beta,\kappa,\mu,\bg_0)$. 

Recall that $\kappa^g$ is defined by $\kappa^g(gA)=\kappa(A)$ for all $A\in G/T$. Set $\wt{\kappa}(A)=\kappa(A^{-1})$.

\begin{corollary}\label{gradings_on_sl_infinity}
Let $\FF$ be an algebraically closed field, $\chr{\FF}\ne 2$, $G$ an abelian group. Suppose $\cL=\Sl(\infty)$ over $\FF$ is given a $G$-grading. Then, as a graded algebra, $\cL$ is isomorphic to one of $\frA^\mathrm{(I)}(G,T,\beta,\kappa)$ or $\frA^\mathrm{(II)}(G,H,h,\beta,\kappa,\mu_0,\bg_0)$. No $G$-graded Lie algebra with superscript \emph{(I)} is isomorphic to one  with superscript \emph{(II)}. Moreover,
\begin{itemize}
\item $\frA^\mathrm{(I)}(G,T,\beta,\kappa)\cong\frA^\mathrm{(I)}(G,T',\beta',\kappa')$ if and only if $T'=T$ and either $\beta'=\beta$ and $\kappa'=\kappa^g$ for some $g\in G$, or $\beta'=\beta^{-1}$ and $\kappa'=\wt{\kappa}^g$ for some $g\in G$.
\item $\frA^\mathrm{(II)}(G,H,h,\beta,\kappa,\mu_0,\bg_0)\cong\frA^\mathrm{(II)}(G,H',h',\beta',\kappa',\mu'_0,\bg'_0)$ if and only if $H'=H$, $h'=h$, $\beta'=\beta$, and there exists $g\in G$ such that $\kappa'=\kappa^\bg$, $\mu'_0=\mu_0\chi^2(g)$ and $\bg'_0=\bg_0\bg^{-2}$.\qed
\end{itemize}
\end{corollary}

\subsection{Orthogonal and symplectic Lie algebras}\label{ssOSLA}

In the case $\cL=\fso(U,\Phi)$ or $\cL=\fsp(U,\Phi)$, we deal with simple Lie algebras of skew-symmetric elements in the associative algebra $\cR=\fpu$ with respect to the involution $\vphi$ determined by the nondegenerate form $\Phi$, which is either orthogonal or symplectic. Here $\Pi$ is identified with $U$ by virtue of $\Phi$. We continue to assume that $U$ is infinite-dimensional. Suppose that $\cL$ is given a $G$-grading. Applying Theorem \ref{T2} to the corresponding Lie homomorphism $\rho\colon\cL\ot\cH\to\cL\ot\cH$ and observing that $\tau=0$ (because $\cL=[\cL,\cL]$), we conclude that  $\rho$ extends to a homomorphism of associative algebras $\rho'\colon\cR\ot\cH\to\cR\ot\cH$. Since $\cL$ generates $\cR$, it follows that $\rho'$ is a comodule structure. This gives $\cR$ a $G$-grading that restricts to the given grading on $\cL$, i.e., $\cL_g=\cR_g\cap\cL$ for all $g\in G$. Moreover, since $\vphi$ restricts to the negative identity on $\cL$, the restriction of the map $\vphi\ot\id_\cH$ to $\cL\ot\cH$ commutes with $\rho$, which implies that $\vphi\ot\id_\cH$ commutes with $\rho'$. This means that $\vphi$ is an involution of $\cR$ as a $G$-graded algebra. Theorem \ref{lfd_2} and Proposition \ref{lfd_inv} classify all the pairs $(\cR,\vphi)$ up to isomorphism.

To state the classification of $G$-gradings on the Lie algebras $\fso(U,\Phi)$ and $\fsp(U,\Phi)$, we introduce the model algebras $\frB(G,T,\beta,\tV,\tB,g_0)$ and $\frC(G,T,\beta,\tV,\tB,g_0)$. 

Let $T\subset G$ be a finite elementary $2$-subgroup with a nondegenerate alternating bicharacter $\beta$. Let $\cL$ be the Lie algebra of skew-symmetric elements in the $G$-graded associative algebra with involution $\cR=\mathfrak{F}(G,T,\beta,\tV,\tB,g_0)$ where the bilinear form $\tB$ on $\tV$ satisfies the symmetry condition \eqref{eq:def_mu} with $\mu_A$ given by \eqref{eq:mu_for_inv}. If $\sgn(\vphi)=1$, we will denote $\cL$ by $\frB(G,T,\beta,\tV,\tB,g_0)$. By Theorem \ref{lfd_2}, it is isomorphic to $\fso(U,\Phi)$ where $U=\tV^\ell$, $\ell^2=|T|$, and the bilinear form $\Phi\colon U\times U\to\FF$ is given by \eqref{eq:two_bilinear_forms_2}. If $\sgn(\vphi)=-1$, we will denote $\cL$ by $\frC(G,T,\beta,\tV,\tB,g_0)$. By Theorem \ref{lfd_2}, it is isomorphic to $\fsp(U,\Phi)$ where $U$ and $\Phi$ are as above. 

\begin{theorem}\label{tsosp}
Let $G$ be an abelian group and let $\FF$ be an algebraically closed field, $\chr{\FF}\ne 2$. If an orthogonal or symplectic Lie algebra $\cL$ of finitary linear operators on an infinite-dimensional vector space over $\FF$ is given a $G$-grading, then $\cL$ is isomorphic as a graded algebra to some $\frB(G,T,\beta,\tV,\tB,g_0)$ in the orthogonal case or $\frC(G,T,\beta,\tV,\tB,g_0)$ in the symplectic case. The $G$-graded algebras with parameters $(T,\beta,\tV,\tB,g_0)$ and $(T',\beta',\tV',\tB',g'_0)$ are isomorphic if an only if $T'=T$, $\beta'=\beta$, and $(\tV',\tB',g'_0)\sim(\tV,\tB,g_0)$ as in Definition \ref{df:datum_finitary_2}.
\end{theorem}

\begin{proof}
The first assertion is already proved. Suppose that $\cL$ and $\cL'$ are the $G$-graded Lie algebras as above with parameters $(T,\beta,\tV,\tB,g_0)$ and $(T',\beta',\tV',\tB',g'_0)$, respectively. If we have an isomorphism $\cL\to\cL'$ of graded algebras, then it extends to an isomorphism $\psi\colon\cR\to\cR'$ of the corresponding associative algebras with involution (use \cite[Theorem 6.18]{FIbook} or our Theorem \ref{T2} with $\cH=\FF$). Since the gradings on $\cR$ and $\cR'$ are uniquely determined by their restrictions to $\cL$ and $\cL'$, respectively, we conclude that $\psi$ is an isomorphism of graded algebras. It remains to invoke Theorem \ref{lfd_2}. 
\end{proof}

\subsubsection{Graded bases} 

The calculations which we used to find graded bases in the case of Type II gradings on special linear Lie algebras apply also for orthogonal and symplectic Lie algebras, with the following simplifications: we always deal with the group $G$ (hence we omit bars and $\chi$) and $Q$ is $\id$ in the orthogonal case and $-\id$ in the symplectic case. Thus,  $\frB(G,T,\beta,\tV,\tB,g_0)$ is the set of skew-symmetric elements in $\frF(G,\cD,\tV,\tB,g_0)$, where $\tB$ is an ``orthosymplectic'' form on $\tV$ that determines a ``Hermitian'' form on $V$, as described in Subsection \ref{ssI}. The skew-symmetric elements
\begin{equation}\label{echiO}
v\ot X_t\ot w-\beta(t)(w\ot X_t\ot v)
\end{equation}
span $\cL=\frB(G,T,\beta,\tV,\tB,g_0)$, so we obtain a basis for $\cL_g$ by letting $A$ range over $G/T$, $v$ over a basis of $\tV_A$, and $w$ over a basis of $\tV_{g_0^{-1}gA^{-1}}$, while $t=g\gamma(A)^{-1}\gamma(g^{-1}A)\tau(g^{-1}A)$, and taking the nonzero elements given by \eqref{echiO}. 

In a similar way, $\frC(G,T,\beta,\tV,\tB,g_0)$ is the set of skew-symmetric elements in $\frF(G,\cD,\tV,\tB,g_0)$, where $\tB$ is an ``orthosymplectic'' form on $\tV$ that determines a ``skew-Hermitian'' form on $V$. The skew-symmetric elements
\begin{equation}\label{echiS}
v\ot X_t\ot w+\beta(t)(w\ot X_t\ot v)
\end{equation}
span $\cL=\frC(G,T,\beta,\tV,\tB,g_0)$, and we obtain a basis for $\cL_g$ as above, but using \eqref{echiS} instead of \eqref{echiO}.

\subsubsection{Countable case}

In the cases $\cL=\So(\infty)$ and $\cL=\Sp(\infty)$, we can express the classification of $G$-gradings in combinatorial terms. Here we have to deal with the pairs $(\cR,\vp)$ where $\cR=M_\infty(\FF)$ is endowed with a $G$-grading and an involution $\vp$ that respects this grading. Such pairs are classified in Corollary \ref{gradings_on_M_infinity_inv}. Namely, $(\cR,\vphi)$ is isomorphic to some $\frF(G,T,\beta,\kappa,\delta,g_0)$ where $|\kappa|$ must be infinite and $\delta=\sgn(\vphi)$. 

We denote by $\frB(G,T,\beta,\kappa,g_0)$ and $\frC(G,T,\beta,\kappa,g_0)$ the $G$-graded Lie algebras of skew-symmetric elements (with respect to $\vphi$) in $\frF(G,T,\beta,\kappa,\delta,g_0)$ where $\delta=1$ and $\delta=-1$, respectively.

\begin{corollary}\label{gradings_on_sop_infinity}
Let $\FF$ be an algebraically closed field, $\chr{\FF}\ne 2$, $G$ an abelian group. Let $\cL$ be $\So(\infty)$, respectively $\Sp(\infty)$, over $\FF$. Suppose $\cL$ is given a $G$-grading. Then, as a graded algebra, $\cL$ is isomorphic to some $\frB(G,T,\beta,\kappa,g_0)$, respectively $\frC(G,T,\beta,\kappa,g_0)$. No $G$-graded Lie algebra $\frB(G,T,\beta,\kappa,g_0)$ is isomorphic to $\frC(G,T',\beta',\kappa',g_0')$. Moreover,
$\frB(G,T,\beta,\kappa,g_0)$ and $\frB(G,T',\beta',\kappa',g_0')$ are isomorphic if and only if  $T'=T$, $\beta'=\beta$, and there exists $g\in G$ such that $g'_0=g_0 g^{-2}$ and $\kappa'=\kappa^g$. The same holds for $\frC(G,T,\beta,\kappa,g_0)$ and $\frC(G,T',\beta',\kappa',g_0')$. \qed
\end{corollary}

%-----------------------------------Bibliography-------------------------------------------

\end{document}